\documentclass[11pt,a4paper,reqno]{amsart}
\usepackage[hmargin=3cm,vmargin=3.2cm]{geometry}
\usepackage{amsmath}
\usepackage{amssymb}
\usepackage{amsthm}
\usepackage[mathscr]{eucal}
\usepackage{enumitem}
\usepackage{amsbsy}
\usepackage{graphicx}
\usepackage{color}
\definecolor{darkred}{rgb}{0.6,0.2,0.2}
\usepackage[pdftitle={On the conditional small ball property of multivariate Levy-driven moving average processes},pdfauthor={Mikko S. Pakkanen, Tommi Sottinen, Adil Yazigi},colorlinks=TRUE,allcolors=darkred,bookmarks=false]{hyperref}

\newtheorem{theorem}{Theorem}[section]
\newtheorem{lemma}[theorem]{Lemma}
\newtheorem{prop}[theorem]{Proposition}
\newtheorem{corollary}[theorem]{Corollary}
\theoremstyle{definition}

\newtheorem{definition}[theorem]{Definition}
\newtheorem{example}[theorem]{Example}
\theoremstyle{remark}
\newtheorem{remark}[theorem]{Remark}
\numberwithin{equation}{section}

\newcommand{\ud}{\mathrm{d}}
\newcommand{\eqdefr}{=\mathrel{\mathop:}}
\newcommand{\eqdefl}{\mathrel{\mathop:}=}
\newcommand{\N}{\mathbb{N}}

\newcommand{\R}{\mathbb{R}}
\newcommand{\M}{\mathbb{M}}
\newcommand{\E}{\mathbf{E}}

\newcommand{\prob}{\mathbf{P}}

\DeclareMathOperator{\conv}{conv}
\DeclareMathOperator{\supp}{supp}
\DeclareMathOperator{\esssupp}{ess\,supp}
\DeclareMathOperator{\ran}{ran}
\DeclareMathOperator{\cone}{cone}
\DeclareMathOperator{\cl}{cl}

\DeclareMathOperator{\interior}{int}

\newcommand{\FF}{\mathscr{F}}

\newcommand{\X}{\bs{X}}
\newcommand{\LL}{\bs{L}}
\newcommand{\A}{\mathbf{A}}

\newcommand{\x}{\bs{x}}

\providecommand{\bs}[1]{\boldsymbol{#1}}

\begin{document}

\pdfstringdefDisableCommands{\def\eqref#1{(\ref{#1})}}


\title[Multivariate L\'{e}vy-driven moving average processes]{On the conditional small ball property of multivariate L\'{e}vy-driven moving average processes}
\date{\today}
\subjclass[2010]{Primary 60G10, 60G17; Secondary 60G22, 60G51.}
\keywords{Small ball probability, conditional full support, moving average process, multivariate L\'evy process, convolution determinant, fractional L\'evy process, L\'evy-driven OU process, L\'evy copula, L\'evy mixing, multivariate subordination.}

\author[Pakkanen, M. S.]{Mikko S. Pakkanen}
\address{Mikko S. Pakkanen\\ Department of Mathematics\\ Imperial College London\\ South Kensington Campus\\ London SW7 2AZ\\ United Kingdom and CREATES \\ Aarhus University \\ Denmark}
\email{\href{mailto:m.pakkanen@imperial.ac.uk}{\nolinkurl{m.pakkanen@imperial.ac.uk}}} 

\author[Sottinen, T.]{Tommi Sottinen}
\address{Tommi Sottinen\\ Department of Mathematics and Statistics\\ University of Vaasa\\ P.O. Box 700\\ FIN-65101 Vaasa\\ Finland}
\email{\href{mailto:tommi.sottinen@iki.fi}{\nolinkurl{tommi.sottinen@iki.fi}}} 

\author[Yazigi, A.]{Adil Yazigi}
\address{Adil Yazigi\\ Department of Mathematics and Statistics\\ University of Vaasa\\ P.O. Box 700\\ FIN-65101 Vaasa\\ Finland}
\email{\href{mailto:adil.yazigi@gmail.com}{\nolinkurl{adil.yazigi@gmail.com}}}


\begin{abstract}
We study whether a multivariate L\'evy-driven moving average process can shadow arbitrarily closely any continuous path, starting from the present value of the process, with positive conditional probability, which we call the conditional small ball property. 
Our main results establish the conditional small ball property for L\'evy-driven moving average processes under natural non-degeneracy conditions on the kernel function of the process and on the driving L\'evy process. 
We discuss in depth how to verify these conditions in practice. 
As concrete examples, to which our results apply, we consider fractional L\'evy processes and multivariate L\'evy-driven Ornstein--Uhlenbeck processes. 
\end{abstract}

\thanks{M. S. Pakkanen was partially supported by CREATES (DNRF78), funded by the Danish National Research Foundation, 
 by Aarhus University Research Foundation (project ``Stochastic and Econometric Analysis of Commodity Markets"), and
by the Academy of Finland (project 258042).
T. Sottinen was partially funded by the Finnish Cultural Foundation (National Foundations' Professor Pool).
A. Yazigi was funded by the Finnish Doctoral Programme in Stochastics and Statistics.
}
\maketitle


\section{Introduction}

We consider multivariate L\'evy-driven moving average processes, i.e., $\R^d$-valued stochastic processes $\X=(\X_t)_{t\geqslant 0}$ defined by the stochastic integral 
\begin{equation}\label{eq:X-intro}
\X_t \,=\, 
\int_{-\infty}^t \big(\Phi(t-u) - \Psi(-u) \big)\ud \bs{L}_u, \quad t \geqslant 0.
\end{equation}
Here the driving process $\bs{L} = (\bs{L}_t)_{t\in \R}$ is a two-sided, $d$-dimensional L\'evy process; and $\Phi$ and $\Psi$ are deterministic functions that take values in the space of $d \times d$ matrices. Under some integrability conditions, which will be made precise in Section \ref{ssec:levy-ma} below, the stochastic integral in \eqref{eq:X-intro} exists in the sense of Rajput and Rosi\'nski \cite{Rajput-Rosinski-1989} as a limit in probability.

The process $\bs{X}$ is infinitely divisible and has stationary increments; in the case $\Psi=0$ the process is also stationary. Several interesting processes are special cases of $\bs{X}$, including fractional Brownian motion \cite{Mandelbrot-Van-Ness-1968}, fractional L\'evy processes \cite{Marquardt-2006}, and L\'evy-driven Ornstein--Uhlenbeck (OU) processes \cite{Barndorff-Nielsen-Shephard-2001,Sato-Yamazato-1984}. Various theoretical aspects of L\'evy-driven and Brownian moving average processes, such as semimartingale property, path regularity, stochastic integration, maximal inequalities, and asymptotic behavior of power variations, have attracted a lot of attention recently; see, e.g., \cite{Basse-2008,Basse-Pedersen-2009,Basse-OConnor-Lachieze-Rey-Podolskij-2015,Bender-Knobloch-Oberacker-2015a,Bender-Knobloch-Oberacker-2015b,Bender-Marquardt-2008,Cherny-2001}.

We study in this paper the following theoretical question regarding the infinite-di\-men\-sion\-al conditional distributions of the process $\bs{X}$. Suppose that $0 \leqslant t_0 < T < \infty$ and take a continuous path $\bs{f}: [t_0,T] \rightarrow \R^d$ such that $\bs{f}(t_0) = \bs{X}_{t_0}$. Can $\bs{X}$ shadow $\bs{f}$ within $\varepsilon$ distance on the interval $[t_0,T]$ with positive conditional probability, given the history of the process up to time $t_0$, for any $\varepsilon>0$ and any choice of $\bs{f}$? If the answer is affirmative, we say that $\bs{X}$ has the \emph{conditional small ball property} (CSBP). When $\bs{X}$ is continuous, the CSBP is equivalent to the \emph{conditional full support} (CFS) property, originally introduced by Guasoni et al.\ \cite{Guasoni-Rasonyi-Schachermayer-2008}, in connection with no-arbitrage and superhedging results for asset pricing models with transaction costs. For recent results on the CFS property, see, e.g., \cite{Bender-Sottinen-Valkeila-2011,Cherny-2008,Gasbarra-Sottinen-vanZanten-2011,Pakkanen-2010,Pakkanen-2011}.

In particular, Cherny \cite{Cherny-2008} has proved the CFS property for univariate Brownian moving average processes. The main results of this paper, stated in Section \ref{sec:mainres}, provide a multivariate generalization of Cherny's result and allow for a non-Gaussian multivariate L\'evy process as the driving process. Our first main result, Theorem \ref{Gaussian}, treats the multivariate Gaussian case (where it leads to no loss of generality if we assume that $\bs{X}$ is continuous). Namely, when $\bs{L}$ is a multivariate Brownian motion, we show that $\bs{X}$ has CFS provided that the \emph{convolution determinant} of $\Phi$ (see Definition \ref{convdet-definition}) does not vanish in a neighborhood of zero. The proof of Theorem \ref{Gaussian} is based on a multivariate generalization of Titchmarsh's convolution theorem \cite{Titchmarsh-1926}, due to Kalisch \cite{Kalisch-1963}.

Our second main result, Theorem \ref{Levy}, covers the case where $\bs{L}$ is purely non-Gaussian. Under some regularity conditions, we show that $\bs{X}$ has the CSBP if the non-degeneracy condition on the convolution determinant of $\Phi$, as seen in the Gaussian case, holds and if $\bs{L}$ satisfies a \emph{small jumps} condition on its L\'evy measure: for any $\varepsilon>0$, the origin of $\R^d$ is in the interior of the convex hull of the support of the L\'evy measure, restricted to the origin-centric ball with radius $\varepsilon$. Roughly speaking, this ensures that the L\'evy process $\bs{L}$ can move arbitrarily close to any point in $\R^d$ with arbitrarily small jumps. 
 We prove Theorem \ref{Levy} building on a small deviations result for L\'evy processes due to Simon \cite{Simon-2001}.

In Section \ref{apps}, we discuss in detail how to check the assumptions of Theorems \ref{Gaussian} and \ref{Levy} for concrete processes.
First, we provide tools for checking the non-degeneracy condition on the convolution determinant of $\Phi$ under various assumption on $\Phi$, including the cases where the components of $\Phi$ are regularly varying at zero and where $\Phi$ is of exponential form, respectively. As a result, we can show CFS for (univariate) fractional L\'evy processes and the CSBP for multivariate L\'evy-driven OU processes. Second, we introduce methods of checking the small jumps condition in Theorem \ref{Levy}, concerning the driving L\'evy process $\bs{L}$. We show how to establish the small jumps condition via the polar decomposition of the L\'evy measure of $\bs{L}$.
Moreover, we check the condition for driving L\'evy processes whose dependence structure is specified using multivariate subordination \cite{Barndorff-Nielsen-Pedersen-Sato-2001}, L\'evy copulas \cite{Kallsen-Tankov-2006}, or L\'evy mixing \cite{Barndorff-Nielsen-Perez-Abreu-Thorbjornsen-2013}.

We present the proofs of Theorems \ref{Gaussian} and \ref{Levy} in Section \ref{proofs}. Additionally, we include Appendices \ref{sec:convolution} and \ref{RV-app}, where we review (and prove) Kalisch's multivariate extension of Titchmarsh's convolution theorem and prove two ancillary results on regularly varying functions, respectively. Finally, in Appendix \ref{app:hitting-times} we comment on a noteworthy consequence of the CSBP in relation to hitting times.

\section{Preliminaries and main results}\label{sec:mainres}

\subsection{Notation and conventions}
 Throughout the paper, we use the convention that $\R_+ \eqdefl [0,\infty)$, $\R_{++} \eqdefl (0,\infty)$, and $\N \eqdefl \{1,2,\ldots\}$. For any $m \in \N$ and $n \in \N$, we write $\M_{m, n}$ for the space of real $m \times n$-matrices, using the shorthand $\M_{n}$ for $\M_{n, n}$, and $\mathbb{S}^+_n$ for the cone of symmetric positive semidefinite matrices in $\M_{n}$. As usual, we identify $\M_{n, 1}$ with $\R^n$. We denote by $\langle \cdot,\cdot\rangle$ the Euclidean inner product in $\R^n$. For $A \in \M_{m,n}$, the notation $\| A \|$ stands for the Frobenius norm of $A$, given by $\big(\sum_{i=1}^m \sum_{j=1}^n A^2_{i,j}\big)^{1/2}$, and $A^\top$ for the transpose of $A$.
 
 When $E$ is a subset of some topological space $\mathbb{X}$, we write $\partial E$ for the boundary of $E$, $\interior E$ for the interior of $E$, and $\cl E$ for the closure of $E$ in $\mathbb{X}$. When $E \subset \R^d$, we use $\conv E$ for the convex hull of $E$, which is the convex set obtained as the intersection of all convex subsets of $\R^n$ that contain $E$ as a subset. For any $\bs{x} \in \R^n$ and $\varepsilon>0$, we write $B(\bs{x},\varepsilon) \eqdefl \{ \bs{y} \in \R^n : \|\bs{x}-\bs{y}\|<\varepsilon\}$. Moreover, $\mathscr{S}^{n-1} \eqdefl \{ \bs{y} \in \R^n : \|\bs{y}\|= 1\}$ stands for the unit sphere in $\R^n$.  For measurable $\bs{f} : \R \rightarrow \R^n$, we write $\esssupp \bs{f}$ for the essential support of $f$, which is the smallest closed subset of $\R$ such that $\bs{f}=\bs{0}$ almost everywhere in its complement.
  
 We denote, for any $p>0$, by $L^p_{\mathrm{loc}}(\R_+,\R)$ the family of real-valued measurable functions $g$ defined on $\R_+$  
such that
\begin{equation*}
\int_0^t |g(u)|^p \ud u < \infty \quad \textrm{for all $t \geqslant 0$}, 
\end{equation*}
extending them to the real line by $g(u) \eqdefl 0$, $u <0$, when necessary. 
Moreover, we denote by $L^p_{\mathrm{loc}}(\R_+,\M_{m, n})$  the family of measurable functions 
$G: \R_+ \rightarrow \M_{m, n}$, where the $(i,j)$-th component function $u \mapsto G_{i,j}(u)$ of $G$ belongs to $L^p_{\mathrm{loc}}(\R_+,\R)$ for any $i = 1,\ldots,m$ and $j = 1,\ldots,n$. Similarly, we write $G \in  L^p(\R_+,\M_{m, n})$ if each of the component functions of $G$ belongs to $L^p(\R_+,\R)$. Note that $L^2(\R_+,\M_{m, n}) \subset L^1_{\mathrm{loc}}(\R_+,\M_{m, n})$.

Recall that
 the convolution of $g\in L^1_{\mathrm{loc}}(\R_+,\R)$ and $h\in L^1_{\mathrm{loc}}(\R_+,\R)$, denoted by $g * h$, is a function on $\R_+$ defined via
\begin{equation*}
(g * h)(t) \eqdefl \int_0^t g(t-u) h(u) \ud u, \quad t \geqslant 0,
\end{equation*}
and that $(g,h) \mapsto g * h$ is an associative and commutative binary operation on $L^1_{\mathrm{loc}}(\R_+,\R)$. 
We additionally extend the convolution to matrix-valued functions $G \in L^1_{\mathrm{loc}}(\R_+,\M_{m,r})$ and $H \in L^1_{\mathrm{loc}}(\R_+,\M_{r,n})$, for any $r \in \N$, by defining
\begin{equation*}
(G * H)_{i,j}(t) \eqdefl \sum_{k=1}^r (G_{i,k} * H_{k,j})(t), \quad t \geqslant 0,
\end{equation*}
for all $i = 1,\ldots,m$ and $j = 1,\ldots,n$. It follows from the properties of the convolution of scalar-valued functions that $G * H \in L^1_{\mathrm{loc}}(\R_+,\M_{m,n})$.

When $\mathbb{X}$ is a Polish space (separable topological space with complete metrization), we write $\mathscr{B}(\mathbb{X})$ for the Borel $\sigma$-algebra of $\mathbb{X}$. When $\mu$ is a Borel measure on $\mathbb{X}$,
we denote by $\supp \mu$ the support of $\mu$, which is the set of points $x \in \mathbb{X}$ such that $\mu(A)>0$ for any open $A \subset \mathbb{X}$ such that $x \in A$. Moreover, if $-\infty<u<v<\infty$ and $\bs{x} \in \R^n$, then $C_{\bs{x}}([u,v],\R^n)$ (resp.\ $C^1_{\bs{x}}([u,v],\R^n)$) stands for the family of continuous (resp.\ continuously differentiable) functions $\bs{f}:[u,v]\rightarrow \R^n$ such that $\bs{f}(u) = \bs{x}$. We equip $C_{\bs{x}}([u,v],\R^n)$ with the uniform topology induced by the sup norm. Finally, $D([u,v],\R^n)$ denotes the space of c\`adl\`ag (right continuous with left limits) functions $\bs{f}:[u,v]\rightarrow \R^n$, equipped with the Skorohod topology, and $D_{\bs{x}}([u,v],\R^n)$ its subspace of functions $\bs{f}\in D([u,v],\R^n)$ such that $\bs{f}(u) = \bs{x}$.

\subsection{L\'evy-driven moving average processes}\label{ssec:levy-ma}

Fix $d \in \N$ and let $(\LL_t)_{t \geqslant 0}$ be a L\'evy process in $\R^d$, defined on a probability space $(\Omega,\FF,\prob)$, with characteristic triplet $(\bs{b},\mathfrak{S},\Lambda)$, where $\bs{b} \in \R^d$, $\mathfrak{S} \in \mathbb{S}_d^+$, and $\Lambda$ is a \emph{L\'evy measure} on $\R^d$, that is, a Borel measure on $\R^d$ that satisfies $\Lambda(\{ \bs{0}\})=0$ and
\begin{equation}\label{levy-integrability}
\int_{\R^d} \min\{1,\| \bs{x} \|^2\}\, \Lambda (\ud \bs{x}) < \infty.
\end{equation}
Recall that the process $\LL$ has the \emph{L\'evy--It\^o representation}
\begin{equation*}
\LL_t \eqdefl \bs{b}t+ \Xi^\top \bs{W}_t + \int_0^t \int_{\{\| \bs{x}\| \leqslant 1 \}} \x \, \widetilde{N} ( \ud \x,\ud u) + \int_0^t \int_{\{\| \x\| > 1 \}} \x \, N ( \ud \x,\ud u), \quad t \geqslant 0,
\end{equation*}
where $\Xi \in \mathbb{M}_d$ is such that $\mathfrak{S} = \Xi^\top \Xi$, $\bs{W}=(\bs{W}_t)_{t \geqslant 0}$ is a standard Brownian motion in $\R^d$, $N$ is a Poisson random measure on $\R^d \times [0,\infty)$ with compensator $\nu(\ud \bs{x},\ud u) \eqdefl \Lambda(\ud \bs{x})\ud u$ and $\widetilde{N} \eqdefl N-\nu$ is the compensated Poisson random measure corresponding to $N$. The Brownian motion $\bs{W}$ and the Poisson random measure $N$ are mutually independent. For a comprehensive treatment of L\'evy processes, we refer the reader to the monograph by Sato \cite{Sato-1999}.

Let $(\LL'_t)_{t \geqslant 0}$ be an independent copy of $(\LL_t)_{t \geqslant 0}$. It is well-known that both $(\LL_t)_{t \geqslant 0}$ and $(\LL'_t)_{t \geqslant 0}$ admit c\`adl\`ag modifications (see, e.g., \cite[Theorem 15.1]{Kallenberg-2002}) and we tacitly work with these modifications. We can then extend $\LL$ to a c\`adl\`ag process on $\R$ by
\begin{equation}\label{two-sided}
\tilde{\LL}_t 
\eqdefl \LL_t \mathbf{1}_{[0,\infty)}(t) + \LL'_{-t-}\mathbf{1}_{(-\infty,0)}(t), \quad t \in \R.
\end{equation}
In what follows we will, for the sake of simplicity, identify $\LL_t$ with $\tilde{\LL}_t$ even when $t < 0$.

Let $\Phi : \R \rightarrow \mathbb{M}_d$ and $\Psi : \R \rightarrow \mathbb{M}_d$ be measurable matrix-valued functions such that $\Phi(t) = 0 =\Psi(t)$ for all $t <0$. Define a kernel function
\begin{equation*}
K(t,u)  \eqdefl \Phi(t-u) - \Psi(-u), \quad (t,u) \in \R^2.
\end{equation*}
The key object we study in this paper is a causal, stationary-increment moving average process $\bs{X}=(\bs{X}_t)_{t \geqslant 0}$ driven by $\LL$, which is defined by
\begin{equation}\label{Xdef}
\X_t \eqdefl \int^t_{-\infty} K(t,u) \ud \LL_u, \quad t \geqslant 0.
\end{equation}
The stochastic integral in \eqref{Xdef} is defined in the sense of Rajput and Rosi\'nski \cite{Rajput-Rosinski-1989} (see also Basse-O'Connor et al.\ \cite{Basse-OConnor-Graversen-Pedersen-2014}) as a limit in probability, provided that (see \cite[Corollary 4.1]{Basse-OConnor-Graversen-Pedersen-2014}) for any $i=1,\ldots,d$ and $t \geqslant 0$,
\begin{subequations}
\begin{align}
 &\int_{-\infty}^t \langle \bs{k}_i(t,u), \mathfrak{S}\, \bs{k}_i(t,u)\rangle  \ud u < \infty,\label{Int-1} \\
 &\int_{-\infty}^t \int_{\R^d}\min\big\{1,\langle\bs{k}_i(t,u), \x\rangle^2\big\} \Lambda(\ud \bs{x}) \ud u  <\infty,\label{Int-2}\\
& \int_{-\infty}^t \bigg|\langle \bs{k}_i(t,u), \bs{b}\rangle + \int_{\R^d}\big(\tau_1\big(\langle \bs{k}_i(t,u), \x\rangle\big) - \langle \bs{k}_i(t,u), \tau_d(\x)\rangle\big) \Lambda(\ud \bs{x}) \bigg|\ud u < \infty,\label{Int-3}
\end{align}
\end{subequations}
where $\bs{k}_i(t,u)\in \R^d$ is the $i$-th row vector of $K(t,u)$, $\tau_1(x) \eqdefl x \mathbf{1}_{\{ |x| \leqslant 1 \}}$, and $\tau_d(\x) \eqdefl \x \mathbf{1}_{\{ \|\x\| \leqslant 1\}}$.

\begin{remark}
When $\LL$ is a driftless Brownian motion (that is, $\bs{b}=\bs{0}$ and $\Lambda = 0$), the conditions \eqref{Int-2} and \eqref{Int-3} become vacuous. When $\det(\mathfrak{S})\neq 0$, the condition \eqref{Int-1} implies that $\int_0^t \|K(t,u)\|^2 \ud u<\infty$, which in turn implies that  $\Phi \in L^2_{\mathrm{loc}}(\R_+,\M_d)$.
\end{remark}

In the case where $\E[\|\bs{L}_1\|^2]<\infty$, which is equivalent to the condition \eqref{square-kernel} below, we find a more convenient sufficient condition for integrability:

\begin{lemma}[Square-integrable case]
Suppose that the L\'evy measure $\Lambda$ satisfies
\begin{equation}\label{square-kernel}
\int_{\| \bs{x} \| > 1} \| \bs{x} \|^2 \Lambda (\ud \bs{x}) <\infty.
\end{equation} 
Then conditions \eqref{Int-1}, \eqref{Int-2}, and \eqref{Int-3} are satisfied provided that
\begin{equation}\label{K-int}
\int_{-\infty}^t (\|K(t,u)\| + \|K(t,u)\|^2) \ud u < \infty \quad \textrm{for any $t \geqslant 0$.}
\end{equation}
\end{lemma}

\begin{proof}
It suffices to only check conditions \eqref{Int-2} and \eqref{Int-3}, condition \eqref{Int-1} being evident. Note that \eqref{square-kernel}, together with \eqref{levy-integrability}, implies that $\int_{\R^d} \| \bs{x}\|^2 \Lambda (\ud \bs{x})< \infty$.  Thus, by the Cauchy--Schwarz inequality and \eqref{K-int},
\begin{equation*}
\int_{-\infty}^t \int_{\R^d}\min\big\{1,\langle\bs{k}_i(t,u), \x\rangle^2\big\} \Lambda(\ud \bs{x}) \ud u \leqslant \int_{-\infty}^t \| \bs{k}_i(t,u)\|^2 \ud u \int_{\R^d} \| x\|^2 \Lambda (\ud \bs{x}) <\infty,
\end{equation*}
so \eqref{Int-2} is satisfied. To verify \eqref{Int-3}, we can estimate
\begin{multline}\label{Levy-int2-bound}
\int_{-\infty}^t \bigg|\langle \bs{k}_i(t,u), \bs{b}\rangle + \int_{\R^d}\big(\tau_1\big(\langle \bs{k}_i(t,u), \x\rangle\big) - \langle \bs{k}_i(t,u), \tau_d(\x)\rangle\big) \Lambda(\ud \bs{x}) \bigg|\ud u \\
\begin{aligned}
&\leqslant  \| \bs{b} \| \int_{-\infty}^t \|\bs{k}_i(t,u)\| \ud u + \int_{-\infty}^t\int_{\{\| \bs{x}\| > 1\}}\big|\tau_1\big(\langle \bs{k}_i(t,u), \x\rangle\big)\big| \Lambda(\ud \bs{x}) \ud u \\
& \quad + \int_{-\infty}^t \int_{\{\| \bs{x}\| \leqslant 1\}}\big|\tau_1\big(\langle \bs{k}_i(t,u), \x\rangle\big) - \langle \bs{k}_i(t,u), \x\rangle\big| \Lambda(\ud \bs{x}) \ud u,
\end{aligned}
\end{multline}
where the first term on the r.h.s.\ is finite under \eqref{K-int}. The second term on the r.h.s.\ of \eqref{Levy-int2-bound} can be shown to be finite using Cauchy--Schwarz, viz.,
\begin{equation*}
\int_{-\infty}^t\int_{\{\| \bs{x}\| > 1\}}\big|\tau_1\big(\langle \bs{k}_i(t,u), \x\rangle\big)\big| \Lambda(\ud \bs{x}) \ud u \leqslant \int_{-\infty}^t \| \bs{k}_i(t,u)\| \ud u \int_{\{\| \bs{x}\| > 1\}} \| x\| \Lambda (\ud \bs{x}) <\infty,\end{equation*}
since \eqref{square-kernel} implies that $\int_{\{\| \bs{x}\| > 1\}} \| x\| \Lambda (\ud \bs{x}) <\infty$. Finally, to treat the third term on the r.h.s.\ of \eqref{Levy-int2-bound}, note that $|\tau_1(x)-x| = |x|\mathbf{1}_{\{ |x|>1\}} \leqslant x^2$ for any $x \in \R$. Hence, 
\begin{multline*}
\int_{-\infty}^t \int_{\{\| \bs{x}\| \leqslant 1\}}\big|\tau_1\big(\langle \bs{k}_i(t,u), \x\rangle\big) - \langle \bs{k}_i(t,u), \x\rangle\big| \Lambda(\ud \bs{x}) \ud u \\
\begin{aligned}
& \leqslant \int_{-\infty}^t \int_{\{\| \bs{x}\| \leqslant 1\}} \langle \bs{k}_i(t,u), \x\rangle^2\Lambda(\ud \bs{x}) \ud u \\
&  \leqslant \int_{-\infty}^t \| \bs{k}_i(t,u)\|^2 \ud u \int_{\{\| \bs{x}\| \leqslant 1\}}  \|\x\|^2\Lambda(\ud \bs{x}) < \infty, 
\end{aligned}
\end{multline*}
by Cauchy--Schwarz and \eqref{square-kernel}.
\end{proof}

In the sequel, unless we are considering some specific examples of $K$, we shall tacitly assume that the integrability conditions \eqref{Int-1}, \eqref{Int-2}, and \eqref{Int-3} are satisfied.

The following decomposition of $\bs{X}$ is fundamental; our technical arguments and some of our assumptions rely on it. For any $t_0 \geqslant 0$,
\begin{equation}\label{fundamental}
\bs{X}_t = \bs{X}_{t_0} + \bar{\bs{X}}^{t_0}_t + \bs{A}^{t_0}_t, \quad t \geqslant t_0, 
\end{equation}
where 
\begin{align*}
\bar{\bs{X}}^{t_0}_t & \eqdefl \int_{t_0}^t \Phi(t-u) \ud \bs{L}_u,\\
\bs{A}^{t_0}_t &  \eqdefl \int_{-\infty}^{t_0} \big( \Phi(t-u) - \Phi(t_0 -u)\big) \ud \bs{L}_u.
\end{align*}
Note that for $t_0 \geqslant 0$ it holds that
\begin{equation*}
\Phi(t-u)\mathbf{1}_{(t_0,t)}(u) = K(t,u)\mathbf{1}_{(t_0,t)}(u) \quad \textrm{for any $t \geqslant t_0$ and $u \in \R$.}
\end{equation*}
Thus the integrability conditions \eqref{Int-1}, \eqref{Int-2}, and \eqref{Int-3} ensure that the stochastic integral defining $\bar{\bs{X}}^{t_0}_t$ exists in the sense of \cite{Rajput-Rosinski-1989} and, consequently, $\bs{A}^{t_0}_t$ is well-defined as well.

\subsection{The conditional small ball property}

To formulate our main results, we introduce the conditional small ball property (cf.\ \cite[p.\ 459]{Bender-Sottinen-Valkeila-2008}):

\begin{definition}\label{CSBP-def}
A c\`adl\`ag process $\bs{Y}=(\bs{Y}_t)_{t \in [0,T]}$, with values in $\R^d$, has the \emph{conditional small ball property} (CSBP) with respect to filtration $\mathscr{F}=(\mathscr{F}_t)_{t \in [0,T]}$, if 
\begin{enumerate}[label=(\roman*),ref=\roman*,leftmargin=2.2em]
\item\label{CSBP-1} $\bs{Y}$ is adapted to $\mathscr{F}$,
\item\label{CSBP-2} for any $t_0 \in [0,T)$, $\bs{f} \in C_{\bs{0}}([t_0,T],\R^d)$, and $\varepsilon >0$, 
\begin{equation}\label{eq:CSBP-prob}
\prob \bigg[ \sup_{t \in [t_0,T]} \| \bs{Y}_t-\bs{Y}_{t_0} - \bs{f}(t) \| < \varepsilon \, \bigg|\, \mathscr{F}_{t_0}  \bigg] >0\quad \textrm{almost surely.}  
\end{equation}
\end{enumerate}
If the process $\bs{Y}$ is continuous and satisfies \eqref{CSBP-1} and \eqref{CSBP-2}, then we say that $\bs{Y}$ has \emph{conditional full support} (CFS) with respect to $\mathscr{F}$.
\end{definition}

\begin{remark}\phantomsection \label{CFS-rem}
\begin{enumerate}[label=(\roman*),ref=\roman*,leftmargin=2.2em]
\item The condition \eqref{CSBP-2} of Definition \ref{CSBP-def} has an equivalent formulation (cf.\ \cite[p.\ 459]{Bender-Sottinen-Valkeila-2008}), where the deterministic time $t_0 \in [0,T)$ in \eqref{eq:CSBP-prob} is replaced with any stopping time $\tau$ such that $\prob[\tau < T]>0$ . The equivalence of this, seemingly stronger, formulation with the original one can be shown adapting the proof of \cite[Lemma 2.9]{Guasoni-Rasonyi-Schachermayer-2008}.
\item More commonly, the CFS property is defined via the condition
\begin{equation}\label{CFS-def}
\supp \mathrm{Law}_{\prob}\big( (\bs{Y}_t)_{t \in [t_0,T]}\big| \mathscr{F}_{t_0} \big) = C_{\bs{Y}_0}([t_0,T],\R^d ) \quad \textrm{a.s.\ for any $t_0 \in [0,T)$,}
\end{equation}
where $\mathrm{Law}_{\prob}\big( (\bs{Y}_t)_{t \in [t_0,T]}\big| \mathscr{F}_{t_0} \big)$ stands for the regular conditional law of $(\bs{Y}_t)_{t \in [t_0,T]}$ on $C([t_0,T],\R^d )$ under $\prob$, given $\mathscr{F}_{t_0}$. The equivalence of condition \eqref{CSBP-2} of Definition \ref{CSBP-def} and \eqref{CFS-def} is an obvious extension of \cite[Lemma 2.1]{Pakkanen-2010}.
 We argue that for discontinuous $\bs{Y}$, the natural generalization of the CFS property would be
\begin{equation}\label{CFS-skorohod}
\supp \mathrm{Law}_{\prob}\big( (\bs{Y}_t)_{t \in [t_0,T]}\big| \mathscr{F}_{t_0} \big) = D_{\bs{Y}_0}([t_0,T],\R^d ) \quad \textrm{a.s.\ for any $t_0 \in [0,T)$,}
\end{equation}
where the regular conditional law is now defined on the Skorohod space $D([t_0,T],\R^d )$. The condition \eqref{CSBP-2} in Definition \ref{CSBP-def} does not imply \eqref{CFS-skorohod}, which is why we refer to the property introduced in Definition \ref{CSBP-def} as the CSBP, instead of CFS. The CFS property for discontinuous processes, defined by \eqref{CFS-skorohod}, appears to be considerably more difficult to check than the CSBP; and the question whether (discontinuous) L\'evy-driven moving average processes have CFS is beyond the scope of the present paper. However, in the context of L\'evy processes, the CFS property could be studied using a Skorohod-space support theorem due to Simon \cite[Corollaire 1]{Simon-2001}, relying on the independence and stationarity of increments. 

\item\label{mixing} Suppose that $\bs{Y}$ has the CSBP (resp. CFS) with respect to $\mathscr{F}$. If $\bar{\bs{Y}}=\big(\bar{\bs{Y}}_t\big)_{t \in [0,T]}$ is a continuous process independent of $\bs{Y}$, then the process
\begin{equation*}
\bs{Z}_t \eqdefl \bs{Y}_t + \bar{\bs{Y}}_t, \quad t \in [0,T],
\end{equation*}
has the CSBP (resp. CFS) with respect to its natural filtration. This is a straightforward extension of \cite[Lemma 3.2]{Gasbarra-Sottinen-vanZanten-2011}.

\item If $\bs{Y}=(\bs{Y}_t)_{t \in [0,T]}$ has the CSBP with respect to $\mathscr{F}$, then $\bs{Y}_t$, for any $t \in (0,T]$, has full support in $\R^d$, in the sense that
\begin{equation*}
\supp \mathrm{Law}_{\prob}(\bs{Y}_t) = \R^d.
\end{equation*}
It is also possible to show that $\bs{Y}$ is then able to hit any open subset of $\R^d$ arbitrarily fast after any stopping time with positive conditional probability. We elaborate on this property in Appendix \ref{app:hitting-times}.

\item The CSBP implies the so-called \emph{stickiness} property, introduced by Guasoni \cite{Guasoni-2006}, which is a sufficient condition for some no-arbitrage results on market models with frictions. Guasoni \cite[Proposition 2.1]{Guasoni-2006} showed that any univariate, sticky c\`adl\`ag process is arbitrage-free (as a price process) under proportional transaction costs. More recently, R\`asonyi and Sayit \cite[Proposition 5.3]{Rasonyi-Sayit-2015} have shown that multivariate, sticky c\`adl\`ag processes are arbitrage-free under superlinear frictions.
\end{enumerate}

\end{remark}

In what follows, we work with the increment filtration $\mathscr{F}^{\bs{L},\mathrm{inc}} = \big(\mathscr{F}^{\bs{L},\mathrm{inc}}_t\big)_{t \in \R}$, given by
\begin{equation*}
\mathscr{F}^{\bs{L},\mathrm{inc}}_t \eqdefl \sigma ( \bs{L}_u- \bs{L}_v: -\infty < v < u \leqslant t).
\end{equation*}
If we prove that the moving average process $\bs{X}$ has the CSBP (resp.\ CFS) with respect to $\mathscr{F}^{\bs{L},\mathrm{inc}}$, then also the CSBP (resp.\ CFS) with respect to the (smaller)  augmented natural filtration of $\bs{X}$ follows, by \cite[Lemma 2.2 and Corollary 2.1]{Pakkanen-2010}. However, we are unable to work with the larger filtration
\begin{equation*}
\mathscr{F}^{\bs{L}}_t \eqdefl \sigma ( \bs{L}_u: -\infty < u \leqslant t),\quad t \in \R,
\end{equation*}
since the increments $\bs{L}_u - \bs{L}_v$, $t \leqslant u < v$, are typically not independent of $\mathscr{F}^{\bs{L}}_t$; 
see \cite{Basse-OConnor-Graversen-Pedersen-2014} for a discussion. (This independence property is essential in our arguments.)

It is convenient to treat separately the two cases where the L\'evy process $\bs{L}$ is Gaussian ($\Lambda =0$) and purely non-Gaussian ($\mathfrak{S} = 0$), respectively. However, we stress that our results make it possible to establish the CSBP/CFS also in the general (mixed) case, where the process $\bs{X}$ can be expressed as a sum of two mutually independent moving average processes, with a Brownian motion and a purely non-Gaussian L\'evy processes as the respective drivers; see Remark \ref{CFS-rem}\eqref{mixing}.

Let us consider the Gaussian case first. In this case we may assume, without loss of generality, that the moving average process $\bs{X}$ is continuous --- if $\bs{X}$ were discontinuous, it would have almost surely unbounded trajectories by a result of Belyaev \cite{Belyaev-1961}.

In his paper \cite{Cherny-2008}, Cherny considered the univariate Brownian moving average process
\begin{equation*}
Z_t \eqdefl \int_{-\infty}^t \big(f(t-s) - f(-s)\big) \ud B_s, \quad t \geqslant 0,
\end{equation*}
where $f$ is a measurable function on $\R_+$ that satisfies $\int_{-\infty}^t \big(f(t-s)-f(-s)\big)^2 \ud s <\infty$ for any $t \geqslant 0$ and $(B_t)_{t \in \R}$ is a two-sided standard Brownian motion. He showed, see \cite[Theorem 1.1]{Cherny-2008}, that $(Z_t)_{t \in [0,T]}$ has CFS for any $T>0$, as long as
\begin{equation}\label{uni-condition}
\esssupp f \neq \varnothing.
\end{equation}
A naive attempt to generalize Cherny's result to the multivariate moving average process $\bs{X}$ in the Gaussian case would be build on the assumption that the components of the kernel function $\Phi$ satisfy individually the univariate condition \eqref{uni-condition}. However, this would fail to account for the possibility that the components of $\bs{X}$ may become perfectly dependent, which would evidently be at variance with the CFS property.

It turns out that a suitable multivariate generalization of the condition \eqref{uni-condition} can be formulated using the following concept:
\begin{definition}\label{convdet-definition}
The \emph{convolution determinant} of $G \in L^1_{\mathrm{loc}}(\R_+,\M_n)$  is a real-valued function given by
\begin{equation}\label{def:det}
\mathrm{det}^* (G)(t) \eqdefl \sum_{\sigma \in S_n} \mathrm{sgn}(\sigma) \big(G_{1,\sigma(1)} * \cdots * G_{n,\sigma(n)}\big)(t),\quad t \geqslant 0,
\end{equation}
where $S_n$ stands for the group of permutations of the set $\{1,\ldots,n\}$ and $\mathrm{sgn}(\sigma)$ for the signature of $\sigma \in S_n$. We note that $\mathrm{det}^* (G) \in L^1_{\mathrm{loc}}(\R_+,\R)$ and that the formula \eqref{def:det} is in fact identical to the definition of the ordinary determinant, except that products of scalars are replaced with convolutions of functions therein. 
\end{definition}

\begin{remark}
Unfortunately, the literature on convolution determinants is rather scarce, but convolution determinants are discussed in some books on integral equations; see, e.g., \cite{Asanov-1998}. We review some pertinent properties of the convolution determinant in Appendix \ref{sec:convolution}.
\end{remark}

In the Gaussian case we obtain the following result, which says that the process $\bs{X}$ has CFS, provided that the convolution determinant of the kernel function $\Phi$ does not vanish near zero. We defer the proof of this result to Section \ref{Gaussian-proof}.

\begin{theorem}[Gaussian case]\label{Gaussian}
Suppose that the driving L\'evy process $\LL$ is a non-de\-gen\-er\-ate Brownian motion, that is, $\det(\mathfrak{S})\neq 0$ and $\Lambda = 0$. 
Assume, further, that the processes $\X$ and $\bs{A}^{t_0}$, for any $t_0 \geqslant 0$, are continuous (modulo taking modifications).
If
\begin{equation*}\label{detstar}
0 \in \esssupp\mathrm{det}^* (\Phi), \tag{DET-\textasteriskcentered}
\end{equation*}
then $(\X_t)_{t \in [0,T]}$ has CFS with respect to $\big(\mathscr{F}^{\bs{L},\mathrm{inc}}_t\big)_{t \in [0,T]}$ for any $T>0$.
\end{theorem}

\begin{remark}\begin{enumerate}[label=(\roman*),ref=\roman*,leftmargin=2.2em]
\item One might wonder if it is possible to replace the condition \eqref{detstar} in Theorem \ref{Gaussian} with a slightly weaker condition, analogous to \eqref{uni-condition}, namely, that 
\begin{equation}\label{multi-det-condition}
\esssupp\mathrm{det}^* (\Phi) \neq \varnothing.
\end{equation}
Unfortunately, \eqref{multi-det-condition} does not suffice in general.
For example, let $\Phi(t)=\mathbf{1}_{[1,2]}(t) I_d$ and $\Psi(t)=\mathbf{1}_{[0,1]}(t) I_d$ for $t \in \R_+$, where $I_d \in \mathbb{M}_d$ is the identity matrix. Then
\begin{equation*}
\mathrm{det}^* (\Phi) = \underbrace{\mathbf{1}_{[1,2]} * \cdots * \mathbf{1}_{[1,2]}}_{d},
\end{equation*}
which follows immediately from the definition \eqref{def:det}. One can now show, e.g., using Titchmarsh's convolution theorem (Lemma \ref{Titchmarsh}, below) and induction in $d$, that \eqref{multi-det-condition} holds.
However,
\begin{equation*}
\bs{X}_1 = \int_{-\infty}^1 \big( \mathbf{1}_{[1,2]}(1-s)I_d - \mathbf{1}_{[0,1]}(-s)I_d\big)\ud \bs{L}_s = \int_{-1}^0 \ud \bs{L}_s-\int_{-1}^0 \ud \bs{L}_s= \bs{0},
\end{equation*}
which indicates that $(\bs{X}_t)_{t \in [0,T]}$ cannot have CFS for any $T\geqslant1$.
\item Theorem \ref{Gaussian} can be generalized to multivariate \emph{Brownian semistationary} ($\mathcal{BSS}$) processes, extending \cite[Theorem 3.1]{Pakkanen-2011}, as follows. Let $\bs{Y}=(\bs{Y}_t)_{t \geqslant 0}$ be a continuous process in $\R^d$ and let $\Sigma=(\Sigma_t)_{t \in \R}$ be a measurable process in $\mathbb{M}_d$ such that $\sup_{t \in \R}\E[\|\Sigma_t\|^2]<\infty$, both independent of the driving Brownian motion $\bs{L}$. Then
\begin{equation*}
\bs{Z}_t \eqdefl \bs{Y}_t + \int_{-\infty}^t \Phi(t-s) \Sigma_s \ud \bs{L}_s, \quad t \geqslant 0,
\end{equation*}
defines a $\mathcal{BSS}$ process in $\R^d$, which has a continuous modification if $Q^\Phi_1(h)+ Q^\Phi_2(h) = O(h^r)$, $h \rightarrow 0+$ for some $r>0$, where $Q^\Phi_1(h)$ and $Q^\Phi_2(h)$ are quantities related to the $L^2$ norm and $L^2$ modulus of continuity of $\Phi$, respectively,
defined by \eqref{Q-def1} and \eqref{Q-def2} below.
In this setting we could show, by adapting the proof of Theorem \ref{Gaussian} and the arguments in \cite[pp.\ 583--585]{Pakkanen-2011}, that
if
\begin{equation*}
\mathrm{Leb}(\{t \in [0,T] : \det(\Sigma_t) = 0\}) = 0 \quad \textrm{a.s.,}
\end{equation*} 
then the process $(\bs{Z}_t)_{t \in [0,T]}$ has CFS with respect to its natural filtration for any $T>0$. The assumption that $\Sigma$ and $\bs{Y}$ are independent of $\bs{L}$ could be relaxed somewhat by an obvious multivariate extension of the factor decomposition used in \cite[p.\ 582]{Pakkanen-2011}.
\end{enumerate}
\end{remark}

Let us then look into the non-Gaussian case with a pure jump process as the driver $\bs{L}$. In addition to that the condition \eqref{detstar} continues to hold, it is essential that the gamut of possible jumps of $\bs{L}$ is sufficiently rich. Consider for instance the case where the components of $\bs{L}$ have only positive jumps, $\Psi=0$, and the elements of $\Phi$ are non-negative. It is not difficult to see that the components of the resulting moving average process $\bs{X}$ will then be always non-negative --- an obvious violation of the CSBP.

To avoid such scenarios, we need to ensure, in particular, that $\bs{L}$ can move close to any point in $\R^d$ with arbitrarily small jumps.
To formulate this small jumps condition rigorously, we introduce, for any $\varepsilon>0$, the restriction of the L\'evy measure $\Lambda$ to the ball $B\big(\bs{0},\varepsilon)$ by
\begin{equation*}
\Lambda_\varepsilon(A) \eqdefl \Lambda \big(A \cap B\big(\bs{0},\varepsilon)\big),\quad A \in \mathscr{B}(\R^d).
\end{equation*}
We obtain the following result, which we shall prove in Section \ref{Levy-proof}.

\begin{theorem}[Non-Gaussian case]\label{Levy}
Suppose that the driving L\'evy process $\bs{L}$ is purely non-Gaussian, that is, $\mathfrak{S}=0$, and that the components of $\Phi$ are of finite variation. Assume, further, that $\X$ is c\`adl\`ag and $\bs{A}^{t_0}$, for any $t_0 \geqslant 0$, is continuous (modulo taking modifications). 
If $\Phi$ satisfies \eqref{detstar}, and if 
\begin{equation*}\label{Levy-support}
\bs{0} \in \interior \conv \supp \Lambda_\varepsilon \quad \textrm{for any $\varepsilon>0$,}\tag{JUMPS}
\end{equation*}
then $(\X_t)_{t \in [0,T]}$ has the CSBP with respect to $\big(\mathscr{F}^{\bs{L},\mathrm{inc}}_t\big)_{t \in [0,T]}$ for any $T>0$.
\end{theorem}

\begin{remark}\phantomsection \label{Levy-remark}
\begin{enumerate}[label=(\roman*),ref=\roman*,leftmargin=2.2em]
\item The proof of Theorem \ref{Levy} hinges on the assumption that the components of $\Phi$ are of finite variation. However, we believe that it should be possible to weaken this assumption to boundedness. Rosi\'nski \cite{Rosinski-1989} has shown that the fine properties of the sample paths of $\bs{X}$ are inherited from the fine properties of $\Phi$. As the fine properties of $\bs{X}$ are not actually ``seen'' by the sup norm used in the definition of the CSBP, it seems plausible that the finite variation assumption is immaterial and merely a limitation of the machinery used in the present proof.
\item In the univariate case, $d=1$, the condition \eqref{Levy-support} reduces to the simple requirement (cf.\ \cite[Proposition 1.1]{Aurzada-Dereich-2009}) that
\begin{equation*}\label{Levy-support-1}
\Lambda\big((-\varepsilon,0)\big)>0\quad \textrm{and}\quad \Lambda\big((0,\varepsilon)\big)>0 \quad \textrm{for any $\varepsilon>0$,}\tag{JUMPS$_1$}
\end{equation*}
which evidently rules out all processes with only positive jumps (e.g., Poisson processes) as drivers.
\item\label{finvarcase} The condition \eqref{Levy-support} could be replaced with a weaker, but more technical, condition that would require a similar support property to hold merely in the subspace
\begin{equation*}
\mathbb{H}_\Lambda \eqdefl \bigg\{ \bs{y} \in \R^d : \int_{\{ \| \bs{x}\| \leqslant 1\}} | \langle \bs{x},\bs{y} \rangle | \Lambda(\ud \bs{x}) < \infty \bigg\},
\end{equation*} 
where the jump activity of $\bs{L}$ has finite variation. In particular, if $\bs{L}$ has infinite variation in all directions in $\R^d$ in the sense that $\mathbb{H}_\Lambda=\{ \bs{0}\}$, then \eqref{Levy-support} can be dropped altogether.
 (In fact, $\mathbb{H}_\Lambda=\{ \bs{0}\}$ can be shown to imply \eqref{Levy-support}.)
\end{enumerate}
\end{remark}

\section{Applications and examples}\label{apps}

In this section, we discuss how to verify the conditions \eqref{detstar} and \eqref{Levy-support} that appear in Theorems \ref{Gaussian} and \ref{Levy}, and provide some concrete examples of processes, to which these results can be applied. However, first we look into the path regularity conditions of Theorems \ref{Gaussian} and \ref{Levy}. 

\subsection{Regularity conditions}

We have assumed in Theorems \ref{Gaussian} and \ref{Levy} that the process $\bs{A}^{t_0}$ is continuous for any $t_0 \geqslant 0$ and that $\bs{X}$ is continuous (resp.\ c\`adl\`ag) in Theorem \ref{Gaussian} (resp.\ Theorem \ref{Levy}). Unfortunately, there are no easily applicable, fully general results that could be used to check these conditions, and they need to be established more-or-less on a case-by-case basis.

In the case where $\E[\|\bs{L}_1\|^2]<\infty$ and $\E[\bs{L}_1]=0$, fairly tractable sufficient criteria for these regularity conditions can be given. To this end, define 
\begin{align}
Q^\Phi_1(h) &\eqdefl \int_0^h \| \Phi(u)\|^2 \ud u,\label{Q-def1} \\
Q^\Phi_2(h) & \eqdefl \int_0^\infty \| \Phi(u+h) - \Phi(u) \|^2 \ud u.\label{Q-def2}
\end{align}
Using \cite[Proposition 2.1]{Marquardt-2006}, we find that for any $t_0 \geqslant 0$,
\begin{align}
\E[\|\bs{X}_{t+h}-\bs{X}_t\|^2] & \leqslant \E[\|\bs{L}_1\|^2] \big(Q^\Phi_1(h) +  Q^\Phi_2(h)\big),  & t &\geqslant 0,\quad h \geqslant 0,\label{continuity-1} \\
\E\big[\big\|\bs{A}^{t_0}_{t+h}-\bs{A}^{t_0}_t\big\|^2\big] & \leqslant \E[\|\bs{L}_1\|^2] Q^\Phi_2(h), & t &\geqslant t_0,\quad h \geqslant 0.\label{continuity-2}
\end{align}
Thus, $\bs{A}^{t_0}$ has a continuous modification by the Kolmogorov--Chentsov criterion if $Q^\Phi_2(h) = O(h^{r_2})$, $h \rightarrow 0+$, for some $r_2>1$. If additionally $Q^\Phi_1(h) = O(h^{r_1})$, $h \rightarrow 0+$, for some $r_1>1$, then $\bs{X}$ has a continuous modification as well. (When $\bs{L}$ is Gaussian, $r_i>0$, $i=1,2$, suffices.) 

When $\bs{A}^{t_0}$ is known, a priori, to be continuous for any $t_0 \geqslant 0$, it follows from \eqref{fundamental}, with $t_0=0$, that the process $\bs{X}$ is continuous (resp.\ c\`adl\`ag) provided that
\begin{equation*}
\bar{\bs{X}}^0_t = \int_0^t \Phi(t-u) \ud \bs{L}_u, \quad t \geqslant 0,
\end{equation*}
is continuous (resp.\ c\`adl\`ag). The path regularity of $\bar{\bs{X}}^0$ becomes an intricate question when $\bs{L}$ is purely non-Gaussian.
Then, a necessary condition for $\bar{\bs{X}}^0$ to have almost surely bounded trajectories --- which is also necessary for the c\`adl\`ag property --- is that $\Phi$ is bounded \cite[Theorem 4]{Rosinski-1989}. For $\bar{\bs{X}}^0$ to be continuous, it is necessary that $\Phi$ is continuous and that $\Phi(0)=0$ \cite[Theorem 4]{Rosinski-1989}. While necessary, these two conditions are not sufficient for the continuity of $\bar{\bs{X}}^0$, however; see \cite[Theorem 3.1]{Kwapien-Marcus-Rosinski-2006}.

Basse and Pedersen \cite{Basse-Pedersen-2009} have obtained sufficient conditions for the continuity of the process $\bar{\bs{X}}^0$. In particular, their results ensure that, in the case where the components of $\bs{L}$ are of finite variation ($\mathfrak{S}=0$ and $\mathbb{H}_\Lambda = \R^d$; see Remark
\ref{Levy-remark}\eqref{finvarcase}), $\bar{\bs{X}}^0$ is c\`adl\`ag (and of finite variation) if the elements of $\Phi$ are of finite variation (which is one of the assumptions used in Theorem \ref{Levy}). In the case where $\bs{L}$ is of infinite variation, the corresponding sufficient condition is more subtle; we refer to \cite[Theorem 3.1]{Basse-Pedersen-2009} for details. Finally, we mention the results of Marcus and Rosi\'nski \cite{Marcus-Rosinski-2005} concerning the continuity of infinitely divisible processes using majorizing measures and metric entropy conditions, which could be applied to study the continuity of $\bs{X}$, $\bar{\bs{X}}^{t_0}$, and $\bs{A}^{t_0}$, $t_0 \geqslant 0$.

\subsection{Kernel functions that satisfy~\eqref{detstar}}

\subsubsection{The Mandelbrot--Van Ness kernel function}\label{frac-Levy} Consider the univariate case, $d=1$, where the processes $\bs{X}$ and $\bs{L}$ reduce to univariate processes $X$ and $L$ and the kernel functions $\Phi$ and $\Psi$ to real-valued functions $\phi$ and $\psi$, respectively. Define, for any $H \in (0,1)$,
\begin{equation}\label{power-kernel}
\phi(t) \eqdefl \psi(t) \eqdefl C_H t^{H-\frac{1}{2}}_+, \quad t \in \R,
\end{equation}
where $x_+ \eqdefl \max \{x,0\}$ for any $x \in \R$ and
\begin{equation*}
C_H \eqdefl \frac{\sqrt{2H \sin (\pi H) \Gamma(2H)}}{\Gamma(H+\frac{1}{2})},
\end{equation*}
which is defined using the \emph{gamma function} $\Gamma(t) \eqdefl \int_0^\infty x^{t-1}e^{-x} \ud x$, $t > 0$.
Then
\begin{equation*}
k(t,s) \eqdefl \phi(t-s) - \psi(-s) = C_H \Big( (t-s)^{H-\frac{1}{2}}_+-(-s)^{H-\frac{1}{2}}_+\Big), \quad (t,s) \in \R^2,
\end{equation*}
is the so-called \emph{Mandelbrot--Van Ness} kernel function (introduced in \cite{Mandelbrot-Van-Ness-1968}) of \emph{fractional Brownian motion} (fBm). That is, with a standard Brownian motion as the driver $L$, the univariate moving average process
\begin{equation}\label{fractional-process}
X_t = \int_{-\infty}^t k(t,s) \ud L_s, \quad t  \geqslant 0,
\end{equation} 
is an fBm with Hurst index $H \in (0,1)$.

Eschewing the fBm, which is already known to have CFS (see \cite{Cherny-2008}), we consider the process \eqref{fractional-process} in the case where the driver $L$ is purely non-Gaussian. Such a process $X$ is called a \emph{fractional L\'evy process}, introduced by Marquardt \cite{Marquardt-2006}. It was shown in \cite{Marquardt-2006}, that if $H\in (\frac{1}{2},1)$, $\E[L^2_1]<\infty$ and $\E[L_1]=0$, then the fractional L\'evy process is well-defined (conditions \eqref{Int-1}, \eqref{Int-2}, and \eqref{Int-3} are satisfied) and has a continuous modification. 
As a consequence of Theorem \ref{Levy}, we obtain:

\begin{corollary}[Fractional L\'evy process] Let $(X_t)_{t \in \R_+}$ be a fractional L\'evy process, given by \eqref{fractional-process}, with $H\in (\frac{1}{2},1)$ and a purely non-Gaussian driving L\'evy process $L$ such that $\E[L^2_1]<\infty$ and $\E[L_1]=0$. If the L\'evy measure $\Lambda$ of $L$ satisfies \eqref{Levy-support-1}, then $(X_t)_{t \in [0,T]}$ has CFS with respect to $\big(\mathscr{F}^{L,\mathrm{inc}}_t\big)_{t \in [0,T]}$ for any $T>0$.
\end{corollary} 
\begin{proof}
The kernel function $\phi$, given by \eqref{power-kernel}, is monotonic and, thus, of finite variation. Additionally, it clearly satisfies the univariate  version of the condition \eqref{detstar}, namely,
\begin{equation}\label{uni-detstar}
0 \in \esssupp \phi.
\end{equation}
The path regularity conditions of Theorem \ref{Levy} can be checked via the criteria \eqref{continuity-1} and \eqref{continuity-2} using 
 \cite[Theorem 4.1]{Marquardt-2006}.\end{proof}

\subsubsection{Regularly varying kernel functions}

The Mandelbrot--Van Ness kernel function can be generalized by retaining the power-law behavior of $\phi$ near zero, but allowing for more general behavior near infinity. This makes it possible to define processes that, say, behave locally like the fBm, in terms of H\"older regularity, but are not long-range dependent. (In effect, this amounts to dispensing with the self-similarity property of the fBm.)

A convenient way to construct such generalizations is to use the concept of \emph{regular variation} (see \cite{Bingham-Goldie-Teugels-1987} for a treatise on the theory of regular variation). Let us recall the basic definitions:

\begin{definition}
\begin{enumerate}[label=(\roman*),ref=\roman*,leftmargin=2.2em]
\item A measurable function $h : \R_+ \rightarrow \R_+$ is \emph{slowly varying} at zero if
\begin{equation*}
\lim_{t \rightarrow 0+} \frac{h(\ell t)}{h(t)} = 1, \quad \textrm{for all $\ell >0$.}
\end{equation*}
\item A measurable function $f : \R_+ \rightarrow \R_+$ is \emph{regularly varying} at zero, with index $\alpha\in \R$, if
\begin{equation*}
\lim_{t \rightarrow 0+} \frac{f(\ell t)}{f(t)} = \ell^\alpha, \quad \textrm{for all $\ell >0$.}
\end{equation*}
We write then $f \in \mathscr{R}_0(\alpha)$.
\end{enumerate}
\end{definition}

\begin{remark}
Clearly, $f \in \mathscr{R}_0(\alpha)$ if and only if $f(t) = t^\alpha h(t)$, $t >0$, for some slowly varying function $h$. Intuitively, $f$ behaves then near zero essentially like the power function $t \mapsto t^\alpha$, as the slowly varying function $h$ varies ``less'' than any power function near zero, in view of \emph{Potter's bounds} (Lemma \ref{potter} in Appendix \ref{RV-app}).
\end{remark}

We discuss now, how to check the condition \eqref{detstar} for a multivariate kernel function $\Phi$, whose components are regularly varying at zero. Checking \eqref{detstar} is then greatly facilitated by the fact that convolution and addition preserve the regular variation property at zero. We prove the first of the following two lemmas in Appendix \ref{RV-app}, while the second follows from an analogous result for regular variation at infinity \cite[Proposition 1.5.7(iii)]{Bingham-Goldie-Teugels-1987}, since $f \in \mathscr{R}_0(\alpha)$ if and only if $x \mapsto f(1/x)$ is regularly varying at infinity with index $-\alpha$; see \cite[pp.\ 17--18]{Bingham-Goldie-Teugels-1987}.

\begin{lemma}[Convolution and regular variation at zero]\label{lem:conv}
Suppose that $f \in \mathscr{R}_0(\alpha) \cap L^1_{\mathrm{loc}}(\R_+,\R_+)$ and $g \in \mathscr{R}_0(\beta) \cap L^1_{\mathrm{loc}}(\R_+,\R_+)$ for some $\alpha>-1$ and $\beta >-1$.
Then,
\begin{equation}\label{rv-conv}
(f * g)(t) \sim \mathrm{B}(\alpha+1,\beta+1) t f(t) g(t), \quad t \rightarrow 0+,
\end{equation}
where $\mathrm{B}(t,u) \eqdefl \int_0^1 s^{u-1} (1-s)^{t-1} \ud s$, $(t,u) \in \R^2_{++}$, is the beta function. Consequently, $f * g \in \mathscr{R}_0(\alpha+\beta+1)$.  
\end{lemma}

\begin{lemma}[Addition and regular variation at zero]\label{lem:add}
If $f \in \mathscr{R}_0(\alpha)$ and $g \in \mathscr{R}_0(\beta)$ for some $\alpha \in \R$ and $\beta \in \R$, then $f+g \in \mathscr{R}_0(\min\{\alpha,\beta\})$.
\end{lemma}

Using Lemmas \ref{lem:conv} and \ref{lem:add}, we can establish \eqref{detstar} for regularly varying multivariate kernel functions under an algebraic constraint on the indices of regular variation.

\begin{prop}[Regularly varying kernels]\label{rv-detstar}
Let $d \geqslant 2$ and let $\Phi = [\Phi_{i,j}]_{(i,j) \in \{1,\ldots,d \}^2} \in L^1_{\mathrm{loc}}(\R_+,\mathbb{M}_d)$ be such that $\Phi_{i,j} \in \mathscr{R}_0(\alpha_{i,j})$ for some $\alpha_{i,j} > -1$ for all $(i,j) \in \{1,\ldots,d \}^2$. Define\begin{equation*}
\alpha^\pm \eqdefl \min \{ \alpha_{1,\sigma(1)}+\cdots+\alpha_{d,\sigma(d)} : \sigma \in S_d,\, \mathrm{sgn}(\sigma)=\pm 1\}.
\end{equation*}
If $\alpha^+ \neq \alpha^-$, then $\Phi$ satisfies \eqref{detstar}.
\end{prop}

\begin{remark}\phantomsection \label{rv-example}
\begin{enumerate}[label=(\roman*),ref=\roman*,leftmargin=2.2em]
\item In the bivariate case, $d=2$, the condition $\alpha^+ \neq \alpha^-$ simplifies to
\begin{equation}\label{bivariate-rv}
\alpha_{1,1} + \alpha_{2,2} \neq \alpha_{1,2} + \alpha_{2,1}.
\end{equation}
\item\label{rv-example-2} The condition $\alpha^+ \neq \alpha^-$ is far from optimal, as it relies on fairly crude information (the indices of regular variation) on the components of $\Phi$. To illustrate this, consider for example
\begin{equation*}
\Phi(t) \eqdefl \begin{bmatrix}
t^{\alpha_{1,1}} & t^{\alpha_{1,2}} \\
t^{\alpha_{2,1}} &  t^{\alpha_{2,2}}
\end{bmatrix}, \quad t >0,
\end{equation*}
where $\alpha_{1,1}, \alpha_{1,2}, \alpha_{2,1}, \alpha_{2,2} > -1$. 
We have
\begin{equation*}
\begin{split}
\mathrm{det}^*(\Phi)(t) & = \mathrm{B}(\alpha_{1,1}+1,\alpha_{2,2}+1) t^{\alpha_{1,1}+\alpha_{2,2}+1} \\ 
& \qquad -\mathrm{B}(\alpha_{2,1}+1,\alpha_{1,2}+1)t^{\alpha_{2,1}+\alpha_{1,2}+1} , \quad t \geqslant 0.
\end{split}
\end{equation*}
Take, say, $\alpha_{1,1} = 2$, $\alpha_{1,2} = 1$, $\alpha_{2,1}= 3$, and $\alpha_{2,2}=2$. Then \eqref{bivariate-rv} does not hold but \eqref{detstar} holds since $\mathrm{B}(3,3) \neq \mathrm{B}(2,4)$.

\item The definition of regular variation dictates that, under the assumptions of Proposition \ref{rv-detstar}, the elements of $\Phi$ must be non-negative, which may be too restrictive. To remove this constraint, it is useful to note that if $\Phi$ satisfies $\eqref{detstar}$ then the kernel $A\Phi(t)$, $t \in \R_+$, for any invertible $A \in \mathbb{M}_d$ also satisfies $\eqref{detstar}$; see Lemma \ref{lem:convdet}\eqref{it:product}.
\end{enumerate} 
\end{remark}

\begin{proof}[Proof of Proposition \ref{rv-detstar}]
Write first
\begin{equation*}
\mathrm{det}^*(\Phi) = \underbrace{\sum_{\sigma \in S_d : \mathrm{sgn}(\sigma)=1} \Phi_{1,\sigma(1)} * \cdots * \Phi_{d,\sigma(d)}}_{\eqdefr \phi^+} - \underbrace{\sum_{\sigma \in S_d : \mathrm{sgn}(\sigma)=-1} \Phi_{1,\sigma(1)} * \cdots * \Phi_{d,\sigma(d)},}_{\eqdefr \phi^-}
\end{equation*}
and note that, by Lemmas \ref{lem:conv} and \ref{lem:add}, we have $\phi^\pm \in \mathscr{R}_0(\alpha^\pm + d-1)$. As a difference of two functions that are regularly varying at zero with different indices, the convolution determinant $\mathrm{det}^*(\Phi)$ cannot vanish in a neighborhood of zero, in view of Potter's bounds (Lemma \ref{potter}, below). Thus, $\Phi$ satisfies \eqref{detstar}.
\end{proof}

\subsubsection{Triangular kernel functions}

When the kernel function $\Phi$ is \emph{upper} or \emph{lower triangular}, the condition \eqref{detstar} becomes very straightforward to check. In fact, it suffices that the diagonal elements of $\Phi$ satisfy the univariate counterpart of \eqref{detstar}.

\begin{prop}[Triangular kernel functions]
Let $d \geqslant 2$ and let $\Phi = [\Phi_{i,j}]_{(i,j) \in \{1,\ldots,d \}^2} \in L^1_{\mathrm{loc}}(\R_+,\mathbb{M}_d)$ be such that $i<j \Rightarrow \Phi_{i,j} = 0$ or $i>j \Rightarrow \Phi_{i,j} = 0$. If
\begin{equation}\label{uni-detstar-2}
0 \in \esssupp \Phi_{i,i} \quad \textrm{for any $i=1,\ldots,d$,}
\end{equation}
then $\Phi$ satisfies \eqref{detstar}.
\end{prop}

\begin{proof}
When $\Phi$ is upper or lower triangular, we find that
\begin{equation*}
\mathrm{det}^*(\Phi) = \Phi_{1,1} * \cdots * \Phi_{d,d},
\end{equation*}
since, in the definition \eqref{def:det}, any summand corresponding to a non-identity permutation $\sigma$ equals zero, as such a summand involves components of $\Phi$ above and below the diagonal. The condition \eqref{detstar} can then be shown to follow from \eqref{uni-detstar-2} using Titchmarsh's convolution theorem (Lemma \ref{Titchmarsh}, below) and induction in $d$.
\end{proof}

\subsubsection{Exponential kernel functions}
In the univariate case, $d=1$ (adopting the notation of Section \ref{frac-Levy}), by setting $\psi = 0$ and
\begin{equation}\label{uni-OU}
\phi(t) \eqdefl e^{at}, \quad t \geqslant 0,
\end{equation}
for some $a<0$, the moving average process $X$ becomes an \emph{Ornstein--Uhlenbeck} (OU) process. It is then clear that $\phi$ satisfies the univariate counterpart \eqref{uni-detstar} of the condition $\eqref{detstar}$.

Multivariate OU processes are defined using the matrix exponential
\begin{equation*}
e^{A} \eqdefl \sum_{j=0}^\infty \frac{1}{j!} A^j,
\end{equation*}
where $A^0 \eqdefl I_d$.
(The matrix exponential $e^A\in \mathbb{M}_d$ is well-defined for any $A \in \mathbb{M}_d$.) More precisely, we define a matrix-valued kernel function $\Phi$ by replacing the parameter $a<0$ in \eqref{uni-OU} with a matrix $A \in \mathbb{M}_d$ whose eigenvalues have strictly negative real parts. Recall that such matrices are called \emph{stable}. We find that such a kernel function $\Phi$ satisfies \eqref{detstar} as well:

\begin{prop}[Exponential kernels]\label{exp-kernel} Suppose that
\begin{equation}\label{exp-Phi}
\Phi(t) = e^{At}, \quad t \geqslant 0,
\end{equation}
for some stable $A \in \mathbb{M}_d$. Then $\Phi \in L^p(\R_+,\mathbb{M}_d)$ for any $p > 0$, and $\Phi$ satisfies \eqref{detstar}.
\end{prop}

\begin{proof}
The assumption that $A$ is stable implies (see \cite[pp.\ 972--973]{Van-Loan-1977}) that there exist constants $c=c(A)>0$ and $\beta=\beta(A)>0$ such that
\begin{equation*}
\|\Phi(t) \| = \|e^{At} \| \leqslant c e^{-\beta t} \quad \textrm{for all $t \geqslant 0$,}
\end{equation*}
which clearly implies that $\Phi \in L^p(\R_+,\mathbb{M}_d)$ for any $p > 0$. 

To show that $\Phi$ satisfies \eqref{detstar}, we consider the Laplace transform $\mathscr{L}[\mathrm{det}^*(\Phi)]$. Note first that since $\Phi \in L^1(\R_+,\mathbb{M}_d)$, each of the components of $\Phi$ belongs to $L^1(\R_+,\R)$, whence $\mathrm{det}^*(\Phi) \in L^1(\R_+,\R)$. Thus, $\mathscr{L}[\mathrm{det}^*(\Phi)](s)$ exists for any $s \geqslant 0$. By the convolution theorem for the Laplace transform, we have
\begin{equation}\label{detstar-det}
\mathscr{L}[\mathrm{det}^*(\Phi)](s) = \det\big(\mathscr{L}[\Phi](s)\big), \quad s \geqslant 0.
\end{equation}
We can now use the well-known fact that the Laplace transform of a matrix-valued function of the form \eqref{exp-Phi} can be expressed using the resolvent of $A$, namely,
\begin{equation}\label{resolvent}
\mathscr{L}[\Phi](s) = (sI_d - A)^{-1}, \quad s \geqslant 0.
\end{equation}
Applying \eqref{resolvent} to \eqref{detstar-det}, we get
\begin{equation*}
\det\big(\mathscr{L}[\Phi](s)\big) = \frac{1}{s^d \det(I_d - s^{-1} A)}, \quad s >0,
\end{equation*}
whence
\begin{equation}\label{detstar-asy}
\mathscr{L}[\mathrm{det}^*(\Phi)](s) \sim \frac{1}{s^d}, \quad s \rightarrow \infty.
\end{equation}

Suppose now that $\Phi$ does not satisfy \eqref{detstar}, which entails that there is $\varepsilon>0$ such that $\mathrm{det}^*(\Phi)(t) = 0$ for almost every $t \in [0,\varepsilon]$. Thus, for any $s \geqslant 0$,
\begin{equation*}
\begin{split}
|\mathscr{L}[\mathrm{det}^*(\Phi)](s)| & \leqslant \int_\varepsilon^\infty e^{-st} |\mathrm{det}^*(\Phi)(t)| \ud t \\
& = e^{-s \varepsilon} \int_0^\infty e^{-su} |\mathrm{det}^*(\Phi)(u+\varepsilon)| \ud u \\
& \leqslant e^{-s \varepsilon} \int_0^\infty |\mathrm{det}^*(\Phi)(u)| \ud u.
\end{split}
\end{equation*}
As $\mathrm{det}^*(\Phi) \in L^1(\R_+,\R)$, we find that its Laplace transform decays exponentially fast, which contradicts \eqref{detstar-asy}.
\end{proof}

\begin{corollary}[L\'evy-driven OU process]
Suppose that $\bs{L}$ satisfies $\E[\|\bs{L}_1\|^2]<\infty$ and $\E[\bs{L}_1]=\bs{0}$. Let $\bs{X}$ be a stationary L\'evy-driven OU process given by
\begin{equation*}
\bs{X}_t \eqdefl \int_{-\infty}^t e^{A(t-s)} \Sigma \ud \bs{L}_s, \quad t \geqslant 0,
\end{equation*}
for some stable $A \in \mathbb{M}_d$ and $\Sigma \in \mathbb{M}_d$ such that $\det(\Sigma) \neq 0$. 
\begin{enumerate}[label=(\roman*),ref=\roman*,leftmargin=2.2em]
\item\label{Levy-OU} If $\mathfrak{S}=0$ and $\Lambda$ satisfies \eqref{Levy-support}, then $(\bs{X}_t)_{t \in [0,T]}$ has the CSBP with respect to $\big(\mathscr{F}^{\bs{L},\mathrm{inc}}_t\big)_{t \in [0,T]}$ for any $T>0$.
\item\label{Gaussian-OU} If $\Lambda=0$ and $\det(\mathfrak{S}) \neq 0$, then $(\bs{X}_t)_{t \in [0,T]}$ has CFS with respect to $\big(\mathscr{F}^{\bs{L},\mathrm{inc}}_t\big)_{t \in [0,T]}$ for any $T>0$
\end{enumerate}
\end{corollary}

\begin{proof}
\eqref{Levy-OU}
The kernel function $\Phi(t) \eqdefl e^{At} \Sigma$, $t \geqslant 0$, is clearly of finite variation and, by Proposition \ref{exp-kernel} and Lemma \ref{lem:convdet}\eqref{it:product}, it satisfies \eqref{detstar}. For any $t_0 \geqslant 0$, we have the decomposition
\begin{equation}\label{AX-decomp}
\bs{X}_t - \bs{X}_{t_0} = \underbrace{(e^{A(t-t_0)}-I_d) \bs{X}_{t_0}}_{= \bs{A}^{t_0}_t} + \underbrace{e^{At} \int_{t_0}^t e^{-As} \Sigma \ud \bs{L}_s}_{= \bar{\bs{X}}^{t_0}_{t}}, \quad t \geqslant t_0.
\end{equation}
Since the map $t \mapsto e^{At}$ is continuous, $\bs{A}^{t_0}$ is a continuous process. Moreover, as a product of a continuous function and a  c\`adl\`ag martingale, $\bar{\bs{X}}^{t_0}$ is c\`adl\`ag, so $\bs{X}$ is c\`adl\`ag as well. The assertion follows then from Theorem \ref{Levy}.

\eqref{Gaussian-OU} Theorem \ref{Gaussian} can be applied here; the continuity of $\bs{X}$ and $\bs{A}^{t_0}$ for any $t_0 \geqslant 0$ can deduced from the decomposition \eqref{AX-decomp}.
\end{proof}

\subsection{L\'evy measures that satisfy \eqref{Levy-support}} 

\subsubsection{Polar decomposition}

When $d \geqslant 2$, the L\'evy measure $\Lambda$ has a polar decomposition (see \cite[Proposition 4.1]{Rosinski-1990} and \cite[Lemma 2.1]{Barndorff-Nielsen-Maejima-Sato-2006}), 
\begin{equation}\label{eq:polar}
\Lambda (A) = \int_{\mathscr{S}^{d-1}} \int_{\R_{++}} \mathbf{1}_A(r\bs{u}) \rho_{\bs{u}}(\ud r) \zeta(\ud \bs{u}),\quad A \in \mathscr{B}(\R^d),
\end{equation}
where $\zeta$ is a finite measure on the unit sphere $\mathscr{S}^{d-1} \subset \R^d$ and $\{ \rho_{\bs{u}} : \bs{u} \in \mathscr{S}^{d-1}\}$ is a family of L\'evy measures on $\R_{++}$ such that the map $\bs{u} \mapsto \rho_{\bs{u}}(A)$ is measurable for any $A \in \mathscr{B}(\R_{++})$. Referring to \eqref{eq:polar}, we say that $\Lambda$ admits a polar decomposition $\{ \zeta, \rho_{\bs{u}} : \bs{u} \in \mathscr{S}^{d-1}\}$.

The condition \eqref{Levy-support} can be established via a polar decomposition as follows:

\begin{prop}[Polar decomposition]\label{polar-supp}
Suppose that $d \geqslant 2$ and that the L\'evy measure $\Lambda$ admits a polar decomposition $\{ \zeta, \rho_{\bs{u}} : \bs{u} \in \mathscr{S}^{d-1}\}$.
If
\begin{enumerate}[label=(\roman*),ref=\roman*,leftmargin=2.2em]
\item\label{item:lambdasupp} $\bs{0} \in \interior \conv \supp \zeta$,
\item\label{item:rhonondeg} $\rho_{\bs{u}}\big( (0,\varepsilon) \big)>0$ for $\zeta$-almost any $\bs{u} \in \mathscr{S}^{d-1}$ and for any $\varepsilon >0$,
\end{enumerate}
then $\Lambda$ satisfies \eqref{Levy-support}.
\end{prop}

\begin{proof}
Suppose that $\Lambda$ does not satisfy \eqref{Levy-support}. Then there exists $\varepsilon>0$ such that either $\bs{0} \in \partial \conv \supp \Lambda_\varepsilon$ or $\bs{0} \notin \conv \supp \Lambda_\varepsilon$. Invoking the supporting hyperplane theorem in the former case and the separating hyperplane theorem in the latter case, we can find a hyperplane $P$, passing through $\bs{0}$, that divides $\R^d$ into two closed half-spaces $E_1$ and $E_2$ such that $E_1 \cap E_2 = P$ and that $\supp \Lambda_\varepsilon \subset E_2$, say. (See Figure \ref{polar-dec-fig} for an illustration.) By property \eqref{item:lambdasupp}, we can find $\bs{v} \in \supp \zeta$ such that $\bs{v} \in \interior E_1$. Moreover, we can find an open neighbourhood $U \subset \mathscr{S}^{d-1}$ of $\bs{v}$ such that $U \subset \interior E_1$ and $\zeta(U)>0$. 
Consider now the truncated cone
\begin{equation*}
C \eqdefl \{ r\bs{u} : r \in (0,\varepsilon),\, \bs{u} \in U \} \subset B(\bs{0},\varepsilon),
\end{equation*}
Clearly, $\mathbf{1}_C (r \bs{u}) = \mathbf{1}_{(0,\varepsilon)}(r) \mathbf{1}_U(\bs{u})$ for any $r>0$ and $\bs{u} \in \mathscr{S}^{d-1}$. Thus, by the polar decomposition, we have that
\begin{equation*}
\begin{split}
\Lambda_\varepsilon(C) = \Lambda(C) & = \int_{\mathscr{S}^{d-1}} \int_{(0,\infty)} \mathbf{1}_{(0,\varepsilon)}(r) \mathbf{1}_U(\bs{u}) \rho_{\bs{u}}(\ud r) \zeta(\ud \bs{u}) \\
& = \int_U \rho_{\bs{u}}\big((0,\varepsilon)\big) \zeta(\ud \bs{u}) >0,
\end{split}
\end{equation*}
where the final inequality follows from property \eqref{item:rhonondeg}. As the set $C$ has positive $\Lambda_\varepsilon$-measure, it must intersect $\supp \Lambda_\varepsilon$, which is a contradiction since $C \subset \interior E_1$ and $\supp \Lambda_\varepsilon \subset E_2$, while $E_2 \cap \, \interior E_1 = \varnothing$.
\end{proof}

\begin{figure}[tbp]
\centering
\includegraphics[scale=0.8]{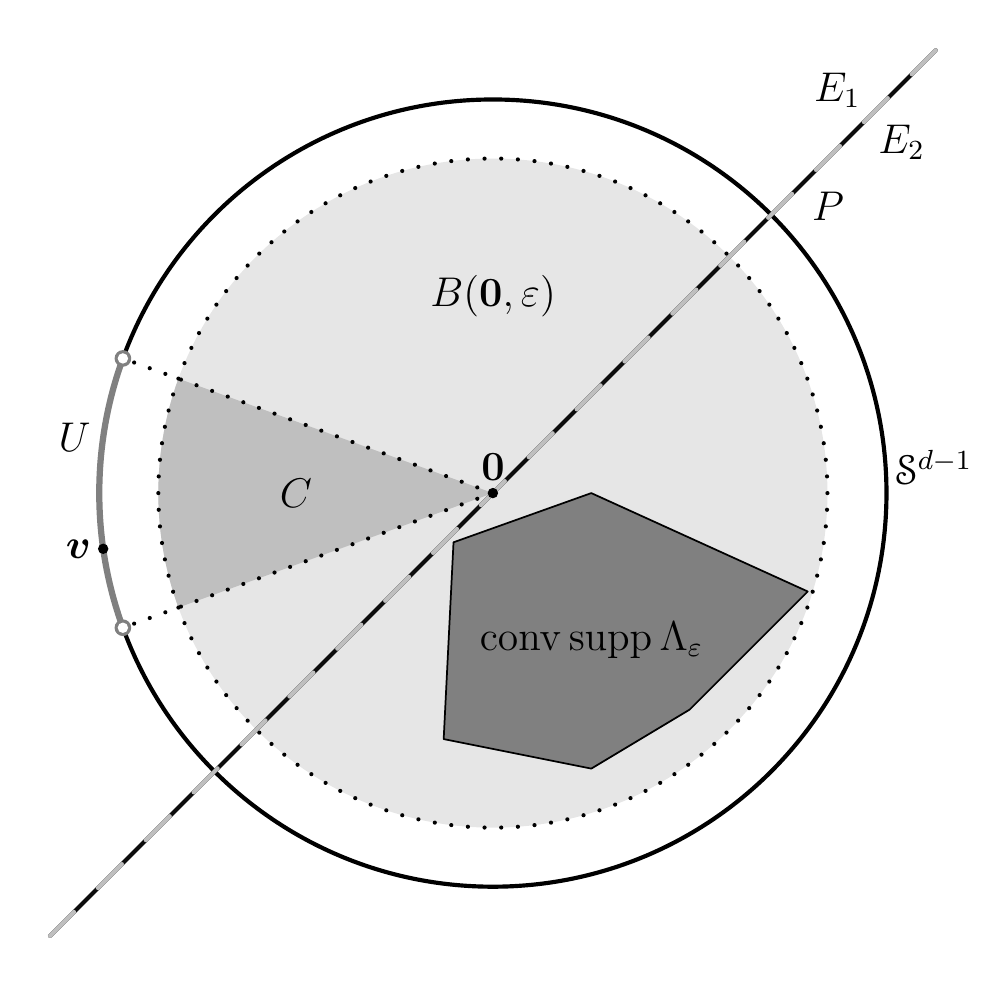}
\caption{The sets defined in the proof of Proposition \ref{polar-supp}.\label{polar-dec-fig}}
\end{figure}

\begin{example}\label{independent-Levy}
Suppose that $\bs{L}$ is a pure jump L\'evy process with mutually independent, non-degenerate components $(L^i_t)_{t \in \R}$, $i = 1,\ldots,d$. Then the L\'evy measure $\Lambda$ of $\bs{L}$ is concentrated on the coordinate axes of $\R^d$; see, e.g., \cite[Exercise 12.10, p.\ 67]{Sato-1999}.
 More concretely, $\Lambda$ admits then a polar decomposition $\{ \zeta, \rho_{\bs{u}} : \bs{u} \in \mathscr{S}^{d-1}\}$, with
\begin{equation*}
\zeta = \sum_{i=1}^d (\delta_{\bs{e}_i}+\delta_{-\bs{e}_i}),
\end{equation*}
where $\delta_{\bs{x}}$ denotes the Dirac measure at $\bs{x} \in \R^d$ and $\bs{e}_i$ is the $i$-th canonical basis vector of $\R^d$, and for any $i = 1,\ldots,d$ and $A \in \mathscr{B}(\R_{++})$,
\begin{equation*}
\rho_{\bs{e}_i}(A) = \lambda_i(A) \quad \textrm{and} \quad \rho_{-\bs{e}_i}(A) = \lambda_i(-A),
\end{equation*}
where $\lambda_i$ denotes the the L\'evy measure of the component $(L^i_t)_{t \in \R}$.

Clearly, $\zeta$ satisfies the condition \eqref{item:lambdasupp} of Proposition \ref{polar-supp} and \eqref{item:rhonondeg} holds provided that the  L\'evy measure $\lambda_i$, for any $i = 1,\ldots,d$, satisfies the condition \eqref{Levy-support-1}, namely,
$\lambda_i\big((-\varepsilon,0)\big)>0$ and $\lambda_i\big((0,\varepsilon)\big)>0$ for any $\varepsilon>0$. 
\end{example}

\begin{example}
It is well known that the L\'evy process $\bs{L}$ is \emph{self-decomposable} (see \cite[Section 3.15]{Sato-1999} for the usual definition) if and only if its L\'evy measure $\Lambda$ admits a polar decomposition $\{ \zeta, \rho_{\bs{u}} : \bs{u} \in \mathscr{S}^{d-1}\}$, where
\begin{equation*}
\rho_{\bs{u}}(\ud r) = \frac{\mathscr{K}({\bs{u}},r)}{r} \ud r, \quad \textrm{$\zeta$-almost every $\bs{u} \in \mathscr{S}^{d-1}$,}
\end{equation*}
for some function $\mathscr{K} : \mathscr{S}^{d-1} \times \R_{++} \rightarrow \R_+$ such that $\mathscr{K}({\bs{u}},r)$ is measurable in $\bs{u}$ and decreasing in $r$ \cite[Theorem 15.10, p.\ 95]{Sato-1999}. For example, for $\alpha \in (0,2)$, setting $\mathscr{K}({\bs{u}},r) \eqdefl r^{-\alpha}$ for all $\bs{u} \in \mathscr{S}^{d-1}$ and $r>0$, we recover the subclass of \emph{$\alpha$-stable processes} \cite[Theorem 14.3, pp.\ 77--78]{Sato-1999}. Clearly, the condition \eqref{item:rhonondeg} in Proposition \ref{polar-supp} is then met if $\sup_{r >0}\mathscr{K}({\bs{u}},r)>0$ for $\zeta$-almost every $\bs{u} \in \mathscr{S}^{d-1}$.

\end{example}

\subsubsection{Multivariate subordination} 
Subordination is a classical method of constructing a new L\'evy process by time-changing an existing L\'evy process with an independent \emph{subordinator}, a L\'evy process with non-decreasing trajectories; see, e.g., \cite[Chapter 6]{Sato-1999}. Barndorff-Nielsen et al.\ \cite{Barndorff-Nielsen-Pedersen-Sato-2001} have introduced a multivariate extension of this technique, which, in particular, provides a way of constructing multivariate L\'evy processes with dependent components (starting from L\'evy processes whose components are mutually independent). 

We discuss now, how the condition \eqref{Levy-support} can be checked when $\Lambda$ is the L\'evy measure of a L\'evy process $\bs{L}$ that has been constructed by multivariate subordination. For simplicity, we consider only a special case of the rather general framework introduced in \cite{Barndorff-Nielsen-Pedersen-Sato-2001}. As the key ingredient in multivariate subordination we take a family of $d\geqslant 2$ mutually independent, univariate L\'evy processes $\big(\widetilde{L}^i_t\big)_{t \geqslant 0}$, $i = 1,\ldots,d$, with respective triplets $\big(\tilde{b}_i,0,\tilde{\lambda}_i\big)$, $i = 1,\ldots,d$. (We eschew Gaussian parts here, as they would be irrelevant in the context of Theorem \ref{Levy}.)
Additionally, we need a $d$-dimensional subordinator $(\bs{T}_t)_{t \geqslant 0}$, independent of $\big(\widetilde{L}^i_t\big)_{t \geqslant 0}$, $i = 1,\ldots,d$, which is a L\'evy process whose triplet $(\bs{c},0,\rho)$ satisfies $\bs{c} = (c_1,\ldots,c_d) \in \R^d_+$ and $\supp \rho \subset \R^d_+$. We denote the components of $\bs{T}_t$ by $T^1_t,\ldots,T^d_t$ for any $t \geqslant 0$.

By \cite[Theorem 3.3]{Barndorff-Nielsen-Pedersen-Sato-2001}, the time-changed process
\begin{equation}\label{subordination}
\bs{L}_t \eqdefl \Big(\widetilde{L}^1_{T^1_t},\ldots,\widetilde{L}^d_{T^d_t}\Big),\quad t \geqslant 0,
\end{equation}
is a L\'evy process with triplet $(\bs{b},0, \Lambda)$ given by
\begin{align}
\bs{b} & \eqdefl \int_{\R_+^d} \rho(\ud \bs{s}) \int_{\{\|\bs{x}\| \leqslant 1 \}} \bs{x} \mu_{\bs{s}}(\ud \bs{x}) + (c_1 \tilde{b}_1,\ldots,c_d \tilde{b}_d),\notag \\
\Lambda(B) &\eqdefl \int_{\R_+^d} \mu_{\bs{s}}(B) \rho(\ud \bs{s}) + \int_B \big(c_1 \mathbf{1}_{\mathscr{A}_1}(\bs{x})\tilde{\lambda}_1(\ud x_1)+\cdots +c_d \mathbf{1}_{\mathscr{A}_d}(\bs{x})\tilde{\lambda}_d(\ud x_d)\big),\, \, B \in \mathscr{B}(\R^d),\label{eq:subtriplet}
\end{align}
where
\begin{equation*}
\mu_{\bs{s}}(\ud \bs{x}) \eqdefl \prob\Big[\widetilde{L}^1_{s_1} \in \ud x_1,\ldots,\widetilde{L}^d_{s_d} \in \ud x_d\Big] = \prod_{j=1}^d \prob\Big[\widetilde{L}^j_{s_j} \in \ud x_j\Big], \quad \bs{s} \in \R_+^d,
\end{equation*}
and $\mathscr{A}_i \eqdefl \{ \bs{x} \in \R^d : x_j = 0,\, j \neq i\}$ is the $i$-th coordinate axis for any $i = 1,\ldots,d$. (While \eqref{subordination} defines only a one-sided process $\bs{L}$, it can obviously be extended to a two-sided L\'evy process using the construction \eqref{two-sided}.) 

\begin{prop}[Multivariate subordination]
Suppose that the L\'evy measure $\Lambda$ is given by \eqref{eq:subtriplet} and that
\begin{equation}\label{marginal-supp}
\tilde{\lambda}_i\big((-\varepsilon,0)\big)>0 \quad \textrm{and} \quad \tilde{\lambda}_i\big((0,\varepsilon)\big)>0 \quad \textrm{for any $i = 1,\ldots,d$ and $\varepsilon>0$.}
\end{equation}
If $\bs{c} \in \R^d_{++}$ or $\R^d_{++} \cap \supp \rho = \varnothing$, then $\Lambda$ satisfies \eqref{Levy-support}.
\end{prop}

\begin{proof}
Let $\bs{i} = (i_1,\ldots,i_d) \in \{0,1\}^d$, $\varepsilon>0$, and $B = B_1\times\cdots \times B_d \subset B(\bs{0},\varepsilon)$, where
\begin{equation*}
B_j \eqdefl \begin{cases}
\big(0,\frac{\varepsilon}{\sqrt{d}}\big), & i_j = 1,\\
\big(-\frac{\varepsilon}{\sqrt{d}},0\big), & i_j = 0,\\
\end{cases}
\end{equation*}
for any $j = 1,\ldots,d$. We observe that
\begin{equation}\label{so-1}
\Lambda_\varepsilon(B)=\Lambda(B) \geqslant \int_B \big(c_1 \mathbf{1}_{\mathscr{A}_1}(\bs{x})\tilde{\lambda}_1(\ud x_1)+\cdots +c_d \mathbf{1}_{\mathscr{A}_d}(\bs{x})\tilde{\lambda}_d(\ud x_d)\big) = \sum_{i=1}^d c_i \tilde{\lambda}(B_i)>0,
\end{equation}
when $\bs{c} \in \R^d_{++}$. The assumption \eqref{marginal-supp} implies, by Lemma \ref{small-ball-property} below, that $\supp \mu_{\bs{s}} = \R^d$, and consequently $\mu_{\bs{s}}(B)>0$, for any $\bs{s} \in \R^d_{++}$. Thus,
\begin{equation}\label{so-2}
\Lambda_\varepsilon(B)=\Lambda(B) \geqslant \int_{\R_{++}^d} \mu_{\bs{s}}(B) \rho(\ud \bs{s}) >0, 
\end{equation}
when $\R^d_{++} \cap \supp \rho = \varnothing$. It follows from \eqref{so-1} and \eqref{so-2}, that $\supp \Lambda_{\varepsilon}$ intersects with any open orthant of $\R^d$, which in turn implies that $\Lambda$ satisfies \eqref{Levy-support}.
\end{proof}

\subsubsection{L\'evy copulas}

Let $d \geqslant 2$. The L\'evy measure $\Lambda$ on $\R^d$ can be defined by specifying $d$ one-dimensional marginal L\'evy measures and describing the dependence structure through a \emph{L\'evy copula}, a concept introduced by Kallsen and Tankov \cite{Kallsen-Tankov-2006}. We provide first a very brief introduction to L\'evy copulas, following \cite{Kallsen-Tankov-2006}.

Below, $\overline{\R} \eqdefl (-\infty,\infty]$. Moreover, for $\bs{a}=(a_1,\ldots,a_d),\, \bs{b}=(b_1,\ldots,b_d) \in \overline{\R}^d$, we write $\bs{a} \leqslant \bs{b}$ (resp.\ $\bs{a} < \bs{b}$) if $a_i \leqslant b_i$ (resp.\ $a_i < b_i$) for any $i = 1,\ldots,d$. For any $\bs{a} \leqslant \bs{b}$ we denote
by $(\bs{a},\bs{b}]$ the hyperrectangle $\prod_{i=1}^d (a_i,b_i]$. 

\begin{definition}
Let $F : \overline{\R}^d \rightarrow \overline{\R}^d$ and let $\bs{a},\, \bs{b} \in \overline{\R}^d$ be such that $\bs{a} \leqslant \bs{b}$. The \emph{rectangular increment} of $F$ over the hyperrectangle $(\bs{a},\bs{b}] \subset \overline{\R}^d$ is defined by
\begin{equation*}
F\big( (\bs{a},\bs{b}]\big) \eqdefl \sum_{\bs{i} \in \{0,1\}^d} (-1)^{d- \sum_{j=1}^d i_j} F\big((1-i_1)a_1 + i_1 b_1,\ldots,(1-i_d)a_d + i_d b_d\big).
\end{equation*}
We say that $F$ is \emph{$d$-increasing} if $F\big( (\bs{a},\bs{b}]\big) \geqslant 0$ for any $\bs{a},\, \bs{b} \in \overline{\R}^d$ such that $\bs{a} \leqslant \bs{b}$. Further, we say that $F$ is \emph{strictly $d$-increasing} if $F\big( (\bs{a},\bs{b}]\big) > 0$ for any $\bs{a},\, \bs{b} \in \overline{\R}^d$ such that $\bs{a} < \bs{b}$.
\end{definition}

\begin{definition}
Let $F : \overline{\R}^d \rightarrow \overline{\R}^d$ be $d$-increasing, with $F(x_1,\ldots,x_d) = 0$ whenever $x_i=0$ for some $i=1,\ldots,d$. For any non-empty $I \subset \{1,\ldots,d\}$, the \emph{$I$-margin} of $F$ is a function $F^I : \overline{\R}^{|I|}\rightarrow \overline{\R}$ given by
\begin{equation*}
F^I\big((x_i)_{i \in I}\big) \eqdefl \lim_{u \rightarrow -\infty} \sum_{(x_i)_{i \in I^c} \in \{u,\infty\}^{|I^c|}} F(x_1,\ldots,x_d) \prod_{i \in I^c} \mathrm{sgn}(x_i),
\end{equation*}
where $\mathrm{sgn}(x) \eqdefl -\mathbf{1}_{(-\infty,0)}(x) + \mathbf{1}_{[0,\infty)}(x)$, $x \in\R$, and $I^c$ stands for the complement $\{1,\ldots,d\} \setminus I$.
\end{definition}

\begin{definition}
A $d$-dimensional \emph{L\'evy copula} is a function $C : \overline{\R}^d \rightarrow \overline{\R}$ such that
\begin{enumerate}[label=(\roman*),ref=\roman*,leftmargin=2.2em]
\item $C(x_1\ldots,x_d) < \infty$ when $x_i < \infty$ for all $i = 1\ldots,d$,
\item\label{lc2} $C(x_1\ldots,x_d) = 0$ when $x_i = 0$ for some $i = 1\ldots,d$,
\item\label{lc3} $C$ is $d$-increasing,
\item\label{lc4} $C^{\{ i\}}(x) = x$ for all $x \in \R$ and $i = 1\ldots,d$.
\end{enumerate}
\end{definition}

\begin{remark} 
The properties \eqref{lc2}, \eqref{lc3}, and \eqref{lc4} of a L\'evy copula are analogous to the defining properties of a classical copula; see, e.g., \cite[Definition 2.10.6]{Nelsen-2006}.
\end{remark}

Recall that a classical copula describes a probability measure on $\R^d$ through its cumulative distribution function, by Sklar's theorem \cite[Theorem 2.10.9]{Nelsen-2006}. In the context of L\'evy copulas, the cumulative distribution function is replaced with the tail integral of the L\'evy measure, which is defined as follows.

\begin{definition}\label{tail-int}
The \emph{tail integral} of the L\'evy measure $\Lambda$ is a function $J_\Lambda : (\R \setminus \{ 0 \})^d \rightarrow \R$ given by
\begin{equation*}
J_\Lambda(x_1,\ldots,x_d) \eqdefl \Lambda\bigg(\prod_{j=1}^d \mathscr{I}(x_j) \bigg) \prod_{i=1}^d \mathrm{sgn}(x_i),
\end{equation*}
where
\begin{equation*}
\mathscr{I}(x) \eqdefl \begin{cases}
(x,\infty), & x \geqslant 0,\\
(-\infty,x], & x < 0.
\end{cases}
\end{equation*}
\end{definition}

\begin{definition}
Let $I \subset \{1,\ldots,d\}$ be non-empty. The \emph{$I$-marginal tail integral} of the L\'evy measure $\Lambda$, denoted by $J^I_\Lambda$, is the tail integral of the L\'evy measure $\Lambda^I$ of the process $\bs{L}^I \eqdefl (L^i)_{i \in I}$, which consists of components of the L\'evy process $\bs{L}$ with L\'evy measure $\Lambda$.
\end{definition}

Kallsen and Tankov \cite[Lemma 3.5]{Kallsen-Tankov-2006} have shown that the L\'evy measure $\Lambda$ is fully determined by its marginal tail integrals $J^I_\Lambda$, $I \subset \{1,\ldots,d\}$ is non-empty. Moreover, they have proved a Sklar's theorem for L\'evy copulas \cite[Theorem 3.6]{Kallsen-Tankov-2006}, which says that for any L\'evy measure $\Lambda$ on $\R^d$, there exists a L\'evy copula $C$ that satisfies
\begin{equation*}
J^I_\Lambda\big((x_i)_{i \in I}\big) = C^I\Big( \big(J_{\Lambda}^{\{ i \}}(x_i)\big)_{i \in I} \Big)
\end{equation*}
for any non-empty $I \subset \{1,\ldots,d\}$ and any $(x_i)_{i \in I} \in (\R \setminus \{ 0\})^{|I|}$. Conversely, we can construct $\Lambda$ from a L\'evy copula $C$ and (one-dimensional) marginal L\'evy measures $\lambda_i$, $i=1,\ldots,d$, by defining 
\begin{equation*}
J^I_\Lambda\big((x_i)_{i \in I}\big) \eqdefl C^I\Big( \big(J_{\lambda_i}(x_i)\big)_{i \in I} \Big)
\end{equation*}
for any non-empty $I \subset \{1,\ldots,d\}$ and any $(x_i)_{i \in I} \in (\R \setminus \{ 0\})^{|I|}$. Then we say that $\Lambda$ has L\'evy copula $C$ and marginal L\'evy measures $\lambda_i$, $i=1,\ldots,d$.

We can now show that $\Lambda$ satisfies \eqref{Levy-support} if it has strictly $d$-increasing L\'evy copula and marginal L\'evy measures such that each of them satisfies \eqref{Levy-support-1}.  

\begin{prop}[L\'evy copula]\label{prop:copula-support}
Suppose that the L\'evy measure $\Lambda$ has L\'evy copula $C$ and marginal L\'evy measures $\lambda_i$, $i=1,\ldots,d$.
If
\begin{enumerate}[label=(\roman*),ref=\roman*,leftmargin=2.2em]
\item\label{item:levydelta} $C$ is strictly $d$-increasing, 
\item\label{item:convsupport} $\lambda_i\big( (-\varepsilon,0)\big) > 0$ and $\lambda_i\big( (0,\varepsilon)\big) > 0$ for any $i = 1,\ldots,d$ and $\varepsilon >0$,
\end{enumerate}
then $\Lambda$ satisfies \eqref{Levy-support}.
\end{prop}

\begin{remark}
Proposition \ref{prop:copula-support} has the caveat that its assumptions do not cover the case where $\Lambda$ is the L\'evy measure of a L\'evy process with independent components, discussed in Example \ref{independent-Levy}. Indeed, the corresponding ``independence L\'evy copula'', characterized in \cite[Proposition 4.1]{Kallsen-Tankov-2006}, is not strictly $d$-increasing. However, it is worth pointing out that it is not sufficient in Proposition \ref{prop:copula-support} that $C$ is merely $d$-increasing, as this requirement is met by all L\'evy copulas, including those that give rise to L\'evy processes with perfectly dependent components; see, e.g., \cite[Theorem 4.4]{Kallsen-Tankov-2006}. 
\end{remark}

\begin{proof}[Proof of Proposition \ref{prop:copula-support}]
Let $\bs{a},\, \bs{b} \in \overline{\R}^d$ be such that $\bs{a} \leqslant \bs{b}$ and $a_ib_i>0$ for any $i = 1,\ldots,d$. Then
\begin{equation}\label{eq:J-incr}
J_\Lambda ( (\bs{a},\bs{b}] )  = \sum_{\bs{i} \in \{0,1\}^d} (-1)^{d- \sum_{j=1}^d i_j} J_\Lambda\big((1-i_1)a_1 + i_1 b_1,\ldots,(1-i_d)a_d + i_d b_d\big),
\end{equation}
where
\begin{multline}\label{eq:J-conv}
J_\Lambda\big((1-i_1)a_1 + i_1 b_1,\ldots,(1-i_d)a_d + i_d b_d\big) \\
\begin{aligned}
 & = C\big( J_{\lambda_1}\big( (1-i_1)a_1 + i_1 b_1\big),\ldots,J_{\lambda_d}\big( (1-i_d)a_d + i_d b_d\big) \big) \\
 & = C\big( i_1 J_{\lambda_1}(b_1)+(1-i_1) J_{\lambda_1}(a_1) ,\ldots,i_d J_{\lambda_d}(b_d)+(1-i_d) J_{\lambda_d}(a_d)  \big).
 \end{aligned}
\end{multline}
It follows from Definition \ref{tail-int} that $x \mapsto J_{\lambda_i}(x)$ is non-increasing on the intervals $(-\infty,0)$ and $(0,\infty)$ for any $i=1,\ldots,d$. Thus,
\begin{equation}\label{eq:tildes}
\tilde{\bs{b}} \eqdefl \big(J_{\lambda_1}(b_1),\ldots, J_{\lambda_d}(b_d)\big) \leqslant  \big(J_{\lambda_1}(a_1),\ldots, J_{\lambda_d}(a_d)\big) \eqdefr \tilde{\bs{a}}.
\end{equation}
Furthermore, by \eqref{eq:J-incr} and \eqref{eq:J-conv},
\begin{equation*}
(-1)^d J_\nu ( (\bs{a},\bs{b}] ) = C\big((\tilde{\bs{b}},\tilde{\bs{a}}] \big). 
\end{equation*}

We establish \eqref{Levy-support} by showing that the intersection of each open orthant of $\R^d$ with $\supp \Lambda_\varepsilon$ is non-empty for any $\varepsilon>0$. To this end, let $\varepsilon>0$, $n \in \N$, and $\bs{i} \in \{0,1\}^d$. Define $\bs{a},\, \bs{b} \in \R^d$ by
\begin{align*}
\bs{a} & \eqdefl \frac{\varepsilon}{\sqrt{d+1}} \bigg(\frac{i_1}{n}-(1-i_1) , \ldots,  \frac{i_d}{n}-(1-i_d) \bigg),\\
\bs{b} & \eqdefl \frac{\varepsilon}{\sqrt{d+1}} \bigg(i_1 -\frac{1-i_1}{n}, \ldots, i_d - \frac{1-i_d}{n}\bigg).
\end{align*}
Note that $\bs{a} \leqslant \bs{b}$ and that $(\bs{a},\bs{b}] \subset B(\bs{0},\varepsilon)$. Moreover, $(\bs{a},\bs{b}]$ belongs to the open orthant
\begin{equation*}
\{(x_1,\ldots,x_d) \in \R^d : (-1)^{1-i_j} x_j> 0,\, j=1,\ldots,d  \},
\end{equation*}
It suffices now to show that $\Lambda ((\bs{a},\bs{b}])>0$ for some $n \in \N$. Since $a_ib_i>0$ for any $i = 1,\ldots,d$, we have (cf.\ the proof of \cite[Lemma 3.5]{Kallsen-Tankov-2006})
\begin{equation*}
\Lambda ((\bs{a},\bs{b}]) = (-1)^d J_\Lambda ( (\bs{a},\bs{b}] ) = C\big((\tilde{\bs{b}},\tilde{\bs{a}}] \big),
\end{equation*}
where $\tilde{\bs{a}}$ and $\tilde{\bs{b}}$ are derived from $\bs{a}$ and $\bs{b}$ via \eqref{eq:tildes}. Observe that
 for any $j = 1,\ldots,d$,
\begin{equation*}
\tilde{a}_j - \tilde{b}_j = J_{\lambda_j}(a_j) - J_{\lambda_j}(b_j) = \begin{cases}
\lambda_j\big(\big(\frac{\varepsilon}{n \sqrt{d+1}},\frac{\varepsilon}{\sqrt{d+1}}\big] \big), & i_j = 1,\\
\lambda_j\big(\big(\frac{-\varepsilon}{\sqrt{d+1}},\frac{-\varepsilon}{n\sqrt{d+1}}\big] \big), & i_j = 0.
\end{cases}
\end{equation*}
It follows now from the assumption \eqref{item:convsupport} that for some sufficiently large $n \geq 2$, we have $\tilde{b}_j < \tilde{a}_j$ for any $j =1,\ldots,d$, that is, $\tilde{\bs{b}} < \tilde{\bs{a}}$. Thanks to the assumption \eqref{item:levydelta}, we can then conclude that $\Lambda ((\bs{a},\bs{b}]) = C\big((\tilde{\bs{b}},\tilde{\bs{a}}] \big) > 0$.
\end{proof}

\begin{example}
We provide here a simple example of a $d$-strictly increasing L\'evy copula, following the construction --- which parallels that of Archimedean copulas \cite[Chapter 4]{Nelsen-2006} in the classical setting --- given by Kallsen and Tankov \cite[Theorem 6.1]{Kallsen-Tankov-2006}. Suppose that $\varphi : [-1,1] \rightarrow [-\infty,\infty]$ is some strictly increasing continuous function such that $\varphi(1)=\infty$, $\varphi(0)=0$, and $\varphi(-1) = -\infty$. Assume, moreover, that $\varphi$ is $d$ times differentiable on $(-1,0)$ and $(0,1)$, so that
\begin{equation}\label{generator}
\frac{\ud^d \varphi(e^x)}{\ud x^d}>0 \quad \textrm{and} \quad \frac{\ud^d \varphi(-e^x)}{\ud x^d}<0 \quad \textrm{for any $x \in (-\infty,0)$.}
\end{equation}
In \cite[Theorem 6.1]{Kallsen-Tankov-2006} it has been shown that the function
\begin{equation*}
C(x_1,\ldots,x_d) \eqdefl \varphi \Bigg( \prod_{i=1}^d \tilde{\varphi}^{-1}(x_i)\Bigg),\quad (x_1,\ldots,x_d) \in \overline{\R}^d,
\end{equation*}
where $\tilde{\varphi}(x) \eqdefl 2^{d-2} \big(\varphi(x)-\varphi(-x)\big)$, $x \in [-1,1]$, is a L\'evy copula. (In fact, in \cite{Kallsen-Tankov-2006} the L\'evy copula property is shown under weaker assumptions with non-strict inequalities in \eqref{generator}.) The argument used in the proof of \cite[Theorem 6.1]{Kallsen-Tankov-2006} to show that $C$ is $d$-increasing translates easily to a proof that $C$ is also strictly $d$-increasing under \eqref{generator}.
One can take, e.g., $\varphi(x) \eqdefl \frac{x}{1-|x|}$, $x \in [-1,1]$, which satisfies the conditions above. In particular,
\begin{equation*}
\begin{aligned}
\frac{\ud^d \varphi(x)}{\ud x^d} & = 
\frac{d!}{(1-x)^{d+1}} > 0, \quad x \in (0,1),\\
(-1)^d \frac{\ud^d \varphi(x)}{\ud x^d} & = -\frac{d!}{(1+x)^{d+1}}<0, \quad x \in (-1,0),
\end{aligned}
\end{equation*}
which ensure that \eqref{generator} holds.
\end{example}

\subsubsection{L\'evy mixing} It is also possible to define new L\'evy measures on $\R^d$, for $d \geqslant 2$, by mixing suitable transformations of a given L\'evy measure on $\R^d$ --- a technique that is called \emph{L\'evy mixing}. We focus here on a particular type of L\'evy mixing, the so-called \emph{Upsilon transformation}, introduced by Barndorff-Nielsen et al.\ \cite{Barndorff-Nielsen-Perez-Abreu-Thorbjornsen-2013} in the multivariate setting. 

The Upsilon transformation amounts to mixing \emph{linear} transformations of a L\'evy measure on $\R^d$. To state the definition of the Upsilon transformation, let $(S,\mathscr{S},\rho)$ be a $\sigma$-finite measure space that parametrizes a family of linear transformations $\R^d \rightarrow \R^d$ via a measurable function $A : S \rightarrow \mathbb{M}_d$. The Upsilon transformation of a L\'evy measure $\Gamma$ on $\R^d$ under $A$, denoted $\Upsilon^0_A[\Gamma]$, is then a Borel measure on $\R^d$ defined by
\begin{equation*}
\Upsilon^0_A[\Gamma](B) \eqdefl \int_S \int_{\R^d} \mathbf{1}_{B \setminus \{ \bs{0}\}}(A(s) \bs{x})  \Gamma(\ud \bs{x})\rho(\ud s), \quad B \in \mathscr{B}(\R^d).
\end{equation*}
Barndorff-Nielsen et al.\ \cite[Theorem 3.3]{Barndorff-Nielsen-Perez-Abreu-Thorbjornsen-2013} have shown that the Upsilon transformation $\Upsilon^0_A$ maps L\'evy measures to L\'evy measures, a crucial property for the validity of this approach, if and only if
\begin{equation}\label{eq:upsilon-Levy}
\rho(\{ s \in S : A(s) \neq 0\})<\infty,\quad
\int_S \| A(s)\|^2 \rho(\ud s)< \infty.
\end{equation}

We show here that the Upsilon transformation preserves the property \eqref{Levy-support}, provided that the function $A$ has a natural non-degeneracy property:\ the matrix $A(s)$ is invertible for any $s$ in a set with positive $\rho$-measure. This result can be applied to construct multivariate L\'evy processes with dependent components for example by applying the Upsilon transformation to the L\'evy measure of a L\'evy process with independent components, which satisfies \eqref{Levy-support} under mild conditions that are easy to check; see Example \ref{independent-Levy}.

\begin{prop}[Upsilon transformation]
Let $(S,\mathscr{S},\rho)$ be a $\sigma$-finite measure space and let $A : (S,\mathscr{S}) \rightarrow \mathbb{M}_d$ be measurable function such that \eqref{eq:upsilon-Levy} holds and
\begin{equation*}
\rho(\{ s \in S : \det (A(s)) \neq 0\} ) > 0.
\end{equation*}
Moreover, let $\Gamma$ be a L\'evy measure on $\R^d$ that satisfies \eqref{Levy-support}. Then the Upsilon transformation $\Lambda \eqdefl \Upsilon^0_A[\Gamma]$ defines a L\'evy measure satisfying \eqref{Levy-support}.
\end{prop}

\begin{proof}
Suppose that $\Lambda = \Upsilon^0_A[\Gamma]$ does not satisfy \eqref{Levy-support}. Then, adapting the early steps seen in the proof of Proposition \ref{polar-supp}, we may find an open half-ball $H(\bs{c},\varepsilon) \eqdefl \{ \bs{x} \in \R^d : \langle \bs{x},\bs{c} \rangle>0,\, \| \bs{x}\|<\varepsilon\}$, for some $\bs{c} \in \R^d \setminus \{ \bs{0}\}$ and $\varepsilon>0$, such that $\Lambda \big(H(\bs{c},\varepsilon)\big)=0$.

Assume for now that $s \in S_A \eqdefl \{ s \in S : \det (A(s)) \neq 0\}$. Then $\| A(s)\bs{x} \| \leqslant \|A(s) \| \| \bs{x}\|$ for any $\bs{x} \in \R^d$, where $\| A(s)\|>0$. So we can deduce that if
\begin{equation*}
\bs{x} \in H\bigg(A(s)^\top \bs{c}, \frac{\varepsilon}{\| A(s)\|}\bigg) = \bigg\{ \bs{x} \in \R^d : \langle \bs{x},A(s)^\top \bs{c} \rangle>0,\, \| \bs{x}\|<\frac{\varepsilon}{\| A(s)\|}\bigg\},
\end{equation*}
then $A(s) \bs{x} \in H(\bs{c},\varepsilon)$. Observe also that $H\big(A(s)^\top \bs{c}, \frac{\varepsilon}{\| A(s)\|}\big)$ is an open half-ball since $A(s)^\top \bs{c} \in \R^d \setminus \{ \bs{0}\}$, due to the property $\det(A(s)^\top) = \det (A(s)) \neq 0$.

Noting that $\bs{0} \notin H(\bs{c},\varepsilon)$, we have thus
\begin{equation}\label{eq:Upsilon-estimate}
\begin{split}
0=\Lambda \big(H(\bs{c},\varepsilon)\big) = \Upsilon^0_A[\Gamma]\big(H(\bs{c},\varepsilon)\big) & = \int_S \int_{\R^d} \mathbf{1}_{H(\bs{c},\varepsilon) \setminus \{ \bs{0}\}}(A(s) \bs{x})  \Gamma(\ud \bs{x})\rho(\ud s) \\
& \geqslant \int_{S_A} \int_{\R^d} \mathbf{1}_{H(\bs{c},\varepsilon)}(A(s) \bs{x})  \Gamma(\ud \bs{x})\rho(\ud s) \\
& \geqslant \int_{S_A}  \Gamma\Bigg(H\bigg(A(s)^\top \bs{c}, \frac{\varepsilon}{\| A(s)\|}\bigg)\Bigg)\rho(\ud s).
\end{split}
\end{equation}
Now for any $s \in S_A$, the open half-ball $H\big(A(s)^\top \bs{c}, \frac{\varepsilon}{\| A(s)\|}\big)$ has positive $\Gamma$-measure since $\Gamma$ satisfies \eqref{Levy-support}. But $\rho(S_A)>0$, by assumption, so in view of \eqref{eq:Upsilon-estimate} we have a contradiction.
\end{proof}

\section{Proofs of the main results}\label{proofs}

\subsection{Proof of Theorem \ref{Gaussian}}\label{Gaussian-proof}

The proof of the CFS property in the Gaussian case follows the strategy used by Cherny \cite{Cherny-2008}, but adapts it to a multivariate setting.

We start with a multivariate extension of a result \cite[Lemma 2.1]{Cherny-2008} regarding density of convolutions.
The proof of \cite[Lemma 2.1]{Cherny-2008} is based on Titchmarsh's convolution theorem (see Lemma \ref{Titchmarsh} in Appendix \ref{sec:convolution}), whereas we use a multivariate extension of Titchmarsh's convolution theorem (Lemma \ref{lem:KT} in Appendix \ref{sec:convolution}), which is due to Kalisch \cite{Kalisch-1963}, to prove the following result.

\begin{lemma}\label{dense-range}
Let $\Phi = [\Phi_{i,j}]_{i,j \in \{1,\ldots,d\}^2}\in L^2_{\mathrm{loc}}(\R_+,\mathbb{M}_d)$ satisfy \eqref{detstar} and let $0\leqslant t_0 < T < \infty$.
Then the linear operator $T_{\Phi} : L^2([t_0,T],\R^d) \rightarrow C_{\bs{0}}([t_0,T],\R^d)$, given by
\begin{equation*}
(T_{\Phi} \bs{f})(t) \eqdefl \int_{t_0}^t \Phi(t-s) \bs{f}(s) \ud s, \quad t \in [t_0,T], \quad \bs{f} \in L^2([t_0,T],\R^d),
\end{equation*}
has dense range.
\end{lemma}

\begin{proof}
By a change of variable, we may assume that $t_0=0$, in which case $T_{\Phi} \bs{f} = \Phi * \bs{f}$.
Analogously to the proof of \cite[Lemma 2.1]{Cherny-2008}, it suffices to show that the range of $T_{\Phi}$, denoted $\ran T_\Phi$, is dense in $L^2([0,T],\R^d)$.

Suppose that, instead, $\cl \ran T_\Phi$, is a strict subset of $L^2([0,T],\R^d)$. Since $\cl \ran T_\Phi$ is a closed linear subspace of $L^2([0,T],\R^d)$, its orthogonal complement $(\cl \ran T_\Phi)^\perp$ is non-trivial. Thus, there exist $\bs{h}=(h_1,\ldots,h_d) \in (\cl \ran T_\Phi)^\perp$ such that 
\begin{equation}\label{h-nondeg}
\esssupp \bs{h} \neq \varnothing.
\end{equation}
Take arbitrary $\bs{f}=(f_1,\ldots,f_d) \in L^2([0,T],\R^d)$. On the one hand, using Fubini's theorem and substitutions $u \eqdefl T-s$ and $v \eqdefl T-t$, we get
\begin{equation*}
\begin{split}
\int_0^T \langle (\Phi * \bs{f})(t), \bs{h}(t)\rangle \ud t & = \int_0^T \sum_{i=1}^n  \bigg(\sum_{j=1}^n \int_0^t \Phi_{i,j} (t-s) f_j(s) \ud s \bigg) h_i(t) \ud t \\
& = \int_0^T \sum_{j=1}^n f_j(T-u) \bigg(\sum_{i=1}^n \int_0^u \Phi_{i,j} (u-v) h_i(T-v) \ud v \bigg)  \ud u \\
& = \int_0^T \langle \bs{f}(T-u) ,(\Phi^\top * \bs{h}(T-\cdot))(u)\rangle\ud u.
\end{split}
\end{equation*}
On the other hand, as $\bs{h} \in (\cl \ran T_\Phi)^\perp$,
\begin{equation*}
\int_0^T \langle(\Phi * \bs{f})(t), \bs{h}(t)\rangle \ud t=\bs{0}. 
\end{equation*}
Thus, $(\Phi^\top * \bs{h}(T-\cdot))(u) = \mathbf{0}$ for almost all $u \in [0,T]$. Since $\mathrm{det}^* (\Phi^\top) = \mathrm{det}^* (\Phi)$, by Lemma \ref{lem:convdet}\eqref{it:transpose}, it follows from Lemma \ref{lem:KT} that $\bs{h}(T-t) = \bs{0}$ for almost all $t \in [0,T]$, contradicting \eqref{h-nondeg}.
\end{proof}

Next we prove a small, but important, result that enables us to deduce the conclusion of Theorem \ref{Gaussian} by establishing an \emph{unconditional} small ball property for the auxiliary process $\bar{\X}^{t_0}$ for any $t_0 \in [0,T)$.
Here, neither Gaussianity nor continuity of $\bs{X}$ is assumed as the result will also be used later in the proof of Theorem \ref{Levy}.
We remark that a similar result, albeit less general, is essentially embedded in the argument that appears in the proof of \cite[Theorem 1.1]{Cherny-2008}.

\begin{lemma}\label{sbp-lemma}
Let $T>0$.
Suppose that $\bs{X}$ is c\`adl\`ag and $\bs{A}^{t_0}$ is continuous for any $t_0 \in [0,T)$. If
\begin{equation}\label{unconditional-sbp}
\prob\bigg[ \sup_{t \in [t_0,T]} \big\|\bar{\X}^{t_0}_t  - \bs{f}(t)\big\| <\varepsilon \bigg]>0 \quad \textrm{for any $t_0 \in [0,T)$, $\bs{f} \in C_{\bs{0}}([t_0,T],\R^d)$ and $\varepsilon>0$,}
\end{equation}
then $(\bs{X}_t)_{t \in [0,T]}$ has the CSBP with respect to $\big(\mathscr{F}^{\bs{L},\mathrm{inc}}_t\big)_{t \in [0,T]}$.
\end{lemma}
\begin{proof}
By construction, $(\bs{X}_t)_{t \in [0,T]}$ is adapted to $\big(\mathscr{F}^{\bs{L},\mathrm{inc}}_t\big)_{t \in [0,T]}$.
We have, by \eqref{fundamental}, for any $t_0 \in [0,T)$, $\bs{f} \in C_{\bs{0}}([t_0,T],\R^d)$, and $\varepsilon>0$, 
\begin{equation}\label{CSBP-prob}
\begin{split}
\prob\bigg[ \sup_{t \in [t_0,T]} \|\X_t   - \X_{t_0} - \bs{f}(t)\| <\varepsilon \, \bigg| \, \mathscr{F}_{t_0}\bigg] 
& = \prob\bigg[ \sup_{t \in [t_0,T]} \big\|\bar{\X}^{t_0}_t + \A^{t_0}_t - \bs{f}(t)\big\| <\varepsilon \, \bigg| \, \mathscr{F}_{t_0}\bigg] \\
& = \E \big[U\big(\bar{\X}^{t_0},\bs{f}-\A^{t_0}\big)\big| \mathscr{F}_{t_0}\big],
\end{split}
\end{equation}
where
\begin{equation*}
U(\bs{g},\bs{h}) \eqdefl \mathbf{1}_{\{\sup_{t \in [t_0,T]} \|\bs{g}(t) -\bs{h}(t) \| < \varepsilon\}}, \quad \bs{g} \in D([t_0,T],\R^d), \quad \bs{h} \in C([t_0,T],\R^d).
\end{equation*}
(Note that $\bar{\X}^{t_0}$ is c\`adl\`ag since $\bs{X}$ is c\`adl\`ag and $\bs{A}^{t_0}$ is continuous.) The space $D([t_0,T],\R^d)$ is Polish, so the random element $\bar{\X}^{t_0}$ in $D([t_0,T],\R^d)$ has a regular conditional law given $\mathscr{F}_{t_0}^{\bs{L},\mathrm{inc}}$. However, as $\bar{\X}^{t_0}$ is independent of $\mathscr{F}_{t_0}^{\bs{L},\mathrm{inc}}$, this conditional law coincides almost surely with the unconditional law $\prob\big[\bar{\X}^{t_0} \in \ud \bs{x}\big]$. 
By the $\mathscr{F}_{t_0}^{\bs{L},\mathrm{inc}}$-measurability of the random element $\A^{t_0}$, the disintegration formula \cite[Theorem 6.4]{Kallenberg-2002} yields
\begin{equation*}
\E \big[U\big(\bar{\X}^{t_0},\bs{f}-\A^{t_0}\big)\big| \mathscr{F}_{t_0}\big] = \int_{D([t_0,T],\R^d)}U\big(\bs{x},\bs{f}-\A^{t_0}\big) \prob\big[\bar{\X}^{t_0} \in \ud \bs{x}\big] \quad \textrm{almost surely.}
\end{equation*}
Evidently,
\begin{equation*}
\prob\bigg[ \sup_{t \in [t_0,T]} \big\|\bar{\X}^{t_0}_t  - \bs{h}(t)\big\| <\varepsilon \bigg] = \int_{D([t_0,T],\R^d)}U\big(\bs{x},\bs{h}\big) \prob\big[\bar{\X}^{t_0} \in \ud \bs{x}\big], \quad \bs{h} \in C([t_0,T],\R^d).
\end{equation*}
So, as $\bs{f}-\A^{t_0} \in C_{\bs{0}}([t_0,T],\R^d)$, the property \eqref{unconditional-sbp} ensures that the conditional probability \eqref{CSBP-prob} is positive almost surely.\end{proof}

\begin{proof}[Proof of Theorem \ref{Gaussian}]
Let $T>0$. Note that when $\bs{X}$ and $\A^{t_0}$ for $t_0 \in [0,T)$ are continuous, then so is $\bar{\bs{X}}^{t_0}$. Thus, the condition \eqref{unconditional-sbp} in Lemma \ref{sbp-lemma} is equivalent to
\begin{equation}\label{full-support}
\supp \mathrm{Law}_\prob\big(\bar{\bs{X}}^{t_0}\big) = C_{\bs{0}}([t_0,T],\R^d) \quad \textrm{for any $t_0 \geqslant 0$.}
\end{equation}
where $\mathrm{Law}_\prob\big(\bar{\bs{X}}^{t_0}\big)$ is understood as the law of $\bar{\bs{X}}^{t_0}$ in $C([t_0,T],\R^d)$.
To prove \eqref{full-support}, define $\bs{f} \in C_{\bs{0}}([t_0,T],\R^d)$ by
\begin{equation}\label{fconv}
\bs{f}(t) \eqdefl \int_{t_0}^T \Phi(t-s) \bs{h}(s) \ud s,
\end{equation} 
for some $\bs{h} \in L^2([t_0,T],\R^d)$. By Girsanov's theorem, there exists $\mathbf{Q} \sim \prob$ such that
\begin{equation}\label{equivalence}
\mathrm{Law}_\prob\big(\bar{\bs{X}}^{t_0}\big) = \mathrm{Law}_\mathbf{Q}\big(\bar{\bs{X}}^{t_0}-\bs{f}\big).
\end{equation}
The support of a probability law on a separable metric space is always non-empty, so there exists $\bs{g} \in \supp \mathrm{Law}_\prob\big(\bar{\bs{X}}^{t_0}\big)$, that is, 
\begin{equation*}
\prob\bigg[ \sup_{t \in [t_0,T]} \big\|\bar{\bs{X}}^{t_0}_t - \bs{g}(t) \big\|<\varepsilon \bigg]>0 \quad \textrm{for any $\varepsilon>0$.}
\end{equation*}
By \eqref{equivalence} and the property $\mathbf{Q} \sim \prob$, we deduce that
\begin{equation*}
\prob\bigg[ \sup_{t \in [t_0,T]} \big\|\bar{\bs{X}}^{t_0}_t - \bs{f}(t)- \bs{g}(t) \big\|<\varepsilon \bigg]>0 \quad \textrm{for any $\varepsilon>0$,}
\end{equation*}
whence $\bs{f}+\bs{g} \in \supp \mathrm{Law}_\prob\big(\bar{\bs{X}}^{t_0}\big)$. 
By Lemma \ref{dense-range}, functions $\bs{f}$ of the form \eqref{fconv} are dense in $C_{\bs{0}}([t_0,T],\R^d)$ under \eqref{detstar}, so the claim \eqref{full-support} follows, as $\supp \mathrm{Law}_\prob\big(\bar{\bs{X}}^{t_0}\big)$ is a closed subset of $C_{\bs{0}}([t_0,T],\R^d)$.
\end{proof}

\subsection{Proof of Theorem \ref{Levy}}\label{Levy-proof}

Let us first recall a result due to Simon \cite{Simon-2001}, which describes the \emph{small deviations} of a general L\'evy process.
In relation to the linear space $\mathbb{H}_\Lambda\subset\R^d$, defined in Remark \ref{Levy-remark}\eqref{finvarcase},
we denote by $\mathrm{pr}_{\mathbb{H}_\Lambda} : \R^d \rightarrow \mathbb{H}_\Lambda$ the orthogonal projection onto $\mathbb{H}_\Lambda$. Further, we set
\begin{equation*}
\bs{a}_{\mathbb{H}_\Lambda} \eqdefl \int_{\{\| \bs{x}\| \leqslant 1\}} \mathrm{pr}_{\mathbb{H}_\Lambda}(\bs{x}) \Lambda(\ud \bs{x}) \in \mathbb{H}_\Lambda.
\end{equation*}
We also recall that the \emph{convex cone} generated by a non-empty set $A \subset \R^d$ is given by
\begin{equation*}
\cone A \eqdefl \{ \alpha_1 \bs{x}^1+\cdots+\alpha_k \bs{x}^k : \alpha_i \geqslant 0,\, \bs{x}^i\in A,\, i=1,\ldots,k,\, k \in \N \}.
\end{equation*}

\begin{theorem}[Simon \cite{Simon-2001}]\label{Small-dev}
Let $\bar{\bs{L}} = \big(\bar{\bs{L}}_t\big)_{t \geqslant 0}$ be a L\'evy process in $\R^d$ with characteristic triplet $\big(\bar{\bs{b}},0,\Lambda\big)$ for some $\bar{\bs{b}} \in \R^d$. Then,
\begin{equation*}
\prob\bigg[ \sup_{t\in [0,T]} \big\|\bar{\bs{L}}_t \big\| < \varepsilon \bigg] >0 \quad \textrm{for any $\varepsilon>0$ and $T>0$,}
\end{equation*}
if and only if
\begin{equation}\label{Levy-drift}
\bs{a}_{\mathbb{H}_\Lambda}-\mathrm{pr}_{\mathbb{H}_\Lambda}(\bar{\bs{b}}) \in \cl \mathscr{B}^\varepsilon_{\mathbb{H}_\Lambda} \quad \textrm{for any $\varepsilon>0$,}
\end{equation}
where
\begin{equation*}
\mathscr{B}^\varepsilon_{\mathbb{H}_\Lambda} \eqdefl \mathrm{pr}_{\mathbb{H}_\Lambda}(\mathscr{B}^\varepsilon),\quad  \mathscr{B}^\varepsilon \eqdefl \cl \cone \supp \Lambda_\varepsilon.
\end{equation*}
\end{theorem}

\begin{remark}
Simon \cite{Simon-2001} allows for a Gaussian part in his result, but we have removed it here, as it would be superfluous for our purposes and as removing it leads to a slightly simpler formulation of the result.
\end{remark}

Theorem \ref{Small-dev} implies that a pure-jump L\'evy process whose L\'evy measure satisfies \eqref{Levy-support} 
has the unconditional small ball property.  While Simon presents a closely related result \cite[Corollaire 1]{Simon-2001} that describes explicitly the support of a L\'evy process in the space of c\`adl\`ag functions, it seems more convenient for our needs to use the following formulation:

\begin{corollary}\label{small-ball-property}
Suppose that the L\'evy process $\bs{L}$ has no Gaussian component, i.e., $\mathfrak{S}=0$, and that its L\'evy measure $\Lambda$ satisfies \eqref{Levy-support}. Then, for any $T>0$, $\bs{f} \in C_{\bs{0}}([0,T],\R^d)$, and $\varepsilon>0$,
\begin{equation*}
\prob\bigg[ \sup_{t \in [0,T]} \| \bs{L}_t - \bs{f}(t)\|<\varepsilon\bigg] >0.
\end{equation*}
\end{corollary}

\begin{proof}
Similarly to the proof of \cite[Corollaire 1]{Simon-2001}, by considering a piecewise affine approximation of $\bs{f}\in C_{\bs{0}}([0,T],\R^d)$ and invoking the independence and stationarity of the increments of $\bs{L}$, it suffices to show that
\begin{equation*}
\prob\bigg[ \sup_{t \in [0,T]} \| \bs{L}_t - \bs{c}t\|<\varepsilon\bigg] >0 \quad \textrm{for any $\bs{c} \in \R^d$, $\varepsilon>0$, and $T>0$.}
\end{equation*}
When $\Lambda$ satisfies \eqref{Levy-support}, there exists $\delta>0$ such that $B(\bs{0},\delta) \subset \conv \supp \Lambda_{\varepsilon}$. Since $\conv \supp \Lambda_{\varepsilon} \subset \cone \supp \Lambda_{\varepsilon}$, it follows that
\begin{equation}\label{cone-supp}
\mathscr{B}^\varepsilon = \R^d.
\end{equation}
(In fact, the conditions \eqref{Levy-support} and \eqref{cone-supp} can be shown to be equivalent.) Consequently, $\mathscr{B}^\varepsilon_{\mathbb{H}_\Lambda} = \mathbb{H}_\Lambda$, and \eqref{Levy-drift} holds for any $\bar{\bs{b}} \in \R^d$. Applying Theorem \ref{Small-dev} to the L\'evy process $\bar{\bs{L}}_t \eqdefl \bs{L}_t - \bs{c}t$, $t \geqslant 0$, which has triplet $(\bar{\bs{b}},0,\Lambda)$ with $\bar{\bs{b}}=\bs{b}-\bs{c}$, completes the proof.
\end{proof}

\begin{proof}[Proof of Theorem \ref{Levy}]
By Lemma \ref{sbp-lemma}, it suffices to show that \eqref{unconditional-sbp} holds.
Let $T>0$, $t_0 \in [0,T)$, $\bs{f} \in C_{\bs{0}}([t_0,T],\R^d)$, and $\varepsilon>0$. Thanks to Lemma \ref{dense-range}, we may in fact take
\begin{equation}\label{f-conv}
\bs{f}(t) \eqdefl \int_{t_0}^t \Phi(t-s)\bs{h}(s) \ud s, \quad t \in [t_0,T],
\end{equation}
for some $\bs{h} \in L^2([t_0,T],\R^d)$, as such functions are dense in $C_{\bs{0}}([t_0,T],\R^d)$. 

Since the components of $\Phi$ are of finite variation, the stochastic integral
\begin{equation}\label{xbar}
\bar{\bs{X}}^{t_0}_t = \int_{t_0}^t \Phi(t-u) \ud \bs{L}_u
\end{equation}
coincides almost surely for any $t \in [t_0,T]$ with an It\^o integral.
The integration by parts formula \cite[Lemma 26.10]{Kallenberg-2002}, applied to both \eqref{f-conv} and \eqref{xbar} yields for fixed $t \in [t_0,T]$,
\begin{equation*}
\bar{\bs{X}}^{t_0}_t -\bs{f}(t) = \Phi_t(t) \bs{L}^{\bs{h}}_t - \Phi_t(t_0) \underbrace{\bs{L}^{\bs{h}}_{t_0}}_{=\bs{0}} - \int_{t_0}^t \ud \Phi_{t}(u) \bs{L}^{\bs{h}}_{u-},
\end{equation*}
where $\bs{L}^{\bs{h}}_u \eqdefl \bs{L}_u - \bs{L}_{t_0} - \int_{t_0}^u \bs{h}(s) \ud s$, $u \in [t_0,T]$ and $\Phi_t(u) \eqdefl \Phi(t-u)$, $u \in [t_0,T]$. Thus, by a standard estimate for Stieltjes integrals,
\begin{equation}\label{TV-bound}
\begin{split}
\big\|\bar{\bs{X}}^{t_0}_t -\bs{f}(t)\big\| & \leqslant \big(\mathscr{V}(\Phi_{t}; [t_0,t]) + \| \Phi_t(t)\|\big) \sup_{u \in [t_0,t]} \big\| \bs{L}^{\bs{h}}_u\big\| \\
& \leqslant \big(\mathscr{V}(\Phi; [0,T]) + \| \Phi(0) \|\big) \sup_{u \in [t_0,T]} \big\| \bs{L}^{\bs{h}}_u\big\|,
\end{split}
\end{equation}
where $\mathscr{V}(G;I)$ denotes the sum of the total variations of the components of a matrix-valued function $G$ on an interval $I \subset \R$. (The assumption that the components of $\Phi$ are of finite variation ensures that both $\mathscr{V}(\Phi; [0,T])$ and $\| \Phi(0) \|$ are finite.)

Finally, by \eqref{TV-bound} and the stationarity of the increments of $\bs{L}$, we have
\begin{equation*}
\begin{split}
\prob \bigg[ \sup_{t \in [t_0,T]} \big\|\bar{\bs{X}}^{t_0}_t -\bs{f}(t)\big\| < \varepsilon \bigg] & \geqslant \prob \bigg[ \sup_{t \in [t_0,T]} \big\| \bs{L}^{\bs{h}}_u\big\| < \tilde{\varepsilon} \bigg] \\
& = \prob \bigg[ \sup_{t \in [0,T-t_0]} \bigg\| \bs{L}_t - \int_{t_0}^{t_0+t} \bs{h}(s) \ud s \bigg\| < \tilde{\varepsilon} \bigg]>0,
\end{split}
\end{equation*}
where $\tilde{\varepsilon} \eqdefl \frac{\varepsilon}{\mathscr{V}(\Phi; [0,T])+ \| \Phi(0) \|} >0$ and where the final inequality follows from Lemma \ref{small-ball-property}.
\end{proof}

\section*{Acknowledgements}

Thanks are due to Ole E. Barndorff-Nielsen,
Andreas Basse-O'Connor,
Nicholas H. Bingham,
Emil Hedevang,
Jan Rosi\'nski, and 
Orimar Sauri
for valuable discussions,
and to Charles M. Goldie and Adam J. Ostaszewski for bibliographic remarks.
M. S. Pakkanen wishes to thank the Department of Mathematics and Statistics at the University of Vaasa for hospitality.

\appendix

\section{Multivariate extension of Titchmarsh's convolution theorem}\label{sec:convolution}

Titchmarsh's convolution theorem is a classical result that describes a connection between the support of a convolution and the supports of the convolved functions:

\begin{theorem}[Titchmarsh \cite{Titchmarsh-1926}]\label{Titchmarsh}
Suppose that $f \in L^1_{\mathrm{loc}}(\R_+,\R)$, $g \in L^1_{\mathrm{loc}}(\R_+,\R)$, and $T>0$. If $(f * g)(t) = 0$ for almost every $t \in [0,T]$, then there exist $\alpha \geqslant 0$ and $\beta\geqslant 0$ such that $\alpha + \beta \geqslant T$, $f(t) = 0$ for almost every $t \in [0,\alpha]$, and $g(t)=0$ for almost every $t \in [0,\beta]$.
\end{theorem}

Titchmarsh's \cite{Titchmarsh-1926} own proof of this result is based on complex analysis; for more elementary proofs relying on real analysis, we refer to \cite{Doss-1988,Yosida-1980}.

In the proof of the crucial Lemma \ref{dense-range} above, we use the following multivariate extension of Theorem \ref{Titchmarsh} that was stated by Kalisch \cite[p.\ 5]{Kalisch-1963}, but without a proof. For the convenience of the reader, we provide a proof below.

\begin{lemma}[Kalisch \cite{Kalisch-1963}]\label{lem:KT}
Suppose that $\Phi = [\Phi_{i,j}]_{(i,j) \in \{1,\ldots,d \}^2} \in L^1_{\mathrm{loc}}(\R_+,\mathbb{M}_d)$ satisfies \eqref{detstar}. If $T>0$ and $\bs{f} = (f_1,\ldots,f_d) \in L^1_{\mathrm{loc}}(\R_+,\R^{d})$ is such that $(\Phi * \bs{f})(t) = \bs{0}$ for almost every $t \in [0,T]$, then $\bs{f}(t) = \bs{0}$ for almost every $t \in [0,T]$.
\end{lemma}

For the proof of Theorem \ref{lem:KT}, we review some properties of the convolution determinant. 
Formally, we may view the convolution determinant as an ``ordinary'' determinant for square matrices whose elements belong to the space $L^1_{\mathrm{loc}}(\R_+,\R)$, equipped with binary operations $+$ and $*$. 
The triple $\big(L^1_{\mathrm{loc}}(\R_+,\R),+,*\big)$ satisfies all the axioms of a \emph{commutative ring} except the requirement that there is a multiplicative identity (the identity element for convolution would be the Dirac delta function). However, any result that has been shown for determinants of matrices whose elements belong to a commutative ring without relying on a multiplicative identity applies to convolution determinants as well. See, e.g., \cite{Loehr-2014} for a textbook where the theory of determinants is developed for matrices whose elements belong to a commutative ring.

Using the aforementioned connection with ordinary determinants, we can deduce some crucial properties of the convolution determinant. Below, $\mathrm{co}^{*}(\Phi;i ,j) \in L^1_{\mathrm{loc}}(\R_+,\mathbb{R}^d)$ for $(i,j) \in \{1,\ldots,d\}^2$ denotes the \emph{$(i,j)$-convolution cofactor} of $\Phi \in L^1_{\mathrm{loc}}(\R_+,\mathbb{M}_d)$, given by the convolution determinant of the $(d-1) \times (d-1)$ matrix obtained from $\Phi$ by deleting the $i$-th row and the $j$-th column, multiplied with the factor $(-1)^{i+j}$. 

\begin{lemma}[Convolution determinants]\label{lem:convdet}
Let $\Phi \in L^1_{\mathrm{loc}}(\R_+,\mathbb{M}_d)$ and $A \in \mathbb{M}_d$. Then,
\begin{enumerate}[label=(\roman*),ref=\roman*,leftmargin=2.2em]
\item\label{it:product} $\mathrm{det}^* (A\Phi) = \mathrm{det}(A)\mathrm{det}^* (\Phi) = \mathrm{det}^* (\Phi A)$,
\item\label{it:transpose} $\mathrm{det}^* (\Phi)=\mathrm{det}^* \big(\Phi^\top\big)$,
\item\label{it:expansion} for any $j$,\ $k = 1,\ldots,d$,
\begin{equation*}
\sum_{i=1}^n \Phi_{i,j} * \mathrm{co}^{*}(\Phi;i,k) = 
\begin{cases}
\mathrm{det}^* (\Phi), & j = k,\\
0, & j \neq k.
\end{cases} 
\end{equation*}
\end{enumerate}
\end{lemma}

\begin{proof} Property \eqref{it:product} can be shown adapting \cite[Section 5.13]{Loehr-2014}; in the proof it is vital to note the distributive property:
\begin{equation*}
\begin{split}
(A\Phi)_{1,\sigma(1)} * \cdots * (A\Phi)_{d,\sigma(d)} & =
\Bigg( \sum_{k_1=1}^d A_{1,k_1} \Phi_{k_1,\sigma(j)}\Bigg) \, * \cdots * \, \Bigg( \sum_{k_d=1}^d A_{d,k_d} \Phi_{k_d,\sigma(d)}\Bigg) \\
& = \sum_{k_1=1}^d\cdots \sum_{k_d=1}^d (A_{1,k_1}\cdots A_{d,k_d}) (\Phi_{k_1,\sigma(1)}* \cdots * \Phi_{k_d,\sigma(d)})
\end{split}
\end{equation*}
for any $\sigma \in S_d$.

Properties \eqref{it:transpose} and \eqref{it:expansion} can be shown following \cite[Section 5.5]{Loehr-2014} and \cite[Section 5.11]{Loehr-2014}, respectively. We note that while an identity matrix appears in \cite[Section 5.11]{Loehr-2014}, in the relevant part of \cite[Section 5.11]{Loehr-2014} it is merely used to express the property \eqref{it:expansion} in a matrix form. 
\end{proof}

\begin{proof}[Proof of Lemma \ref{lem:KT}] Consider the convolution adjugate matrix of $\Phi$,
\begin{equation*}
\mathrm{adj}^* (\Phi) \eqdefl \begin{bmatrix} \mathrm{co}^{*}(\Phi;1,1) & \cdots & \mathrm{co}^{*}(\Phi;d,1) \\
\vdots & \ddots & \vdots \\
\mathrm{co}^{*}(\Phi;1,d) & \cdots & \mathrm{co}^{*}(\Phi;d,d)
\end{bmatrix} \in L^1_{\mathrm{loc}}(\R_+,\mathbb{M}_d).
\end{equation*}
On the one hand, the $r$-th row of $\mathrm{adj}^* (\Phi) * (\Phi * \bs{f})$, for any $r =1,\ldots,d$, can be calculated as
\begin{equation*}
\begin{split}
\sum_{i=1}^n \mathrm{co}^{*}(\Phi;i,r) * \bigg(\sum_{j=1}^n \Phi_{i,j} * f_j\bigg) & = \sum_{j=1}^n \bigg( \sum_{i=1}^n \Phi_{i,j}*\mathrm{co}^{*}(\Phi;i,r)\bigg) * f_j \\
& = \mathrm{det}^* (\Phi) * f_r,
\end{split}
\end{equation*}
using associativity and commutativity of convolutions and invoking Lemma \ref{lem:convdet}\eqref{it:expansion}. On the other hand, since $(\Phi * \bs{f})(t) = \bs{0}$ for almost every $t \in [0,T]$, we have that
\begin{equation*}
\big(\mathrm{adj}^* (\Phi) * (\Phi * \bs{f})\big)(t)=\bs{0}\quad \textrm{for all $t \in [0,T]$.}
\end{equation*}
Thus, $\big(\mathrm{det}^*(\Phi) * f_r\big)(t) = 0$ for any $t \in [0,T]$ and $r = 1,\ldots,d$.
In view of Theorem \ref{Titchmarsh}, we can then conclude that $\bs{f}(t) = \mathbf{0}$ for almost every $t \in [0,T]$. 
\end{proof}

\begin{remark}
The condition \eqref{detstar} is, in general, unrelated with the similar condition
\begin{equation}\label{det-supp}
0 \in \esssupp \mathrm{det} (\Phi),
\end{equation}
where $\mathrm{det} (\Phi)$ stands for the function $\det\big(\Phi(t)\big)$, $t \geqslant 0$. Remark \ref{rv-example}\eqref{rv-example-2} provides an example of a function $\Phi$ that satisfies \eqref{detstar} but not \eqref{det-supp}. For an example that satisfies \eqref{det-supp} but not \eqref{detstar}, consider
\begin{equation*}
\Phi(t) \eqdefl \begin{bmatrix}
2 t & t^2 \\
3 t^2 & t^3
\end{bmatrix}, \quad t \geqslant 0.
\end{equation*}
We have
$\mathrm{det}\big(\Phi(t)\big)= 2t^4 - 3t^4 = -t^4 < 0$ for  $t>0$,
while (cf.\ Remark \ref{rv-example}\eqref{rv-example-2})
\begin{equation*}
\mathrm{det}^*(\Phi)(t) = \big(2\, \mathrm{B}(2,4) - 3 \, \mathrm{B}(3,3)\big) t^{5} = 0, \quad t \geqslant 0,
\end{equation*}
where the beta function $\mathrm{B}$ satisfies $2\, \mathrm{B}(2,4) = 3 \, \mathrm{B}(3,3)$ by \cite[Equation 1.5(6)]{Erdelyi-et-al-1953}.
\end{remark}

\section{Proof of Lemma \ref{lem:conv}}\label{RV-app}

Let us first recall \emph{Potter's bounds}, which enable us to estimate the behavior of a regularly varying function near zero. 

\begin{lemma}[Potter]\label{potter}
Let $f \in \mathscr{R}_0(\alpha)$ for some $\alpha \in \R$. Then for any $\varepsilon>0$ and $c>0$ there exists $t_0=t_0(c,\varepsilon)>0$ such that
\begin{equation*}
\frac{f(t)}{f(u)} \leqslant c \max\bigg\{ \bigg(\frac{t}{u}\bigg)^{\alpha+\varepsilon},\, \bigg(\frac{t}{u}\bigg)^{\alpha-\varepsilon} \bigg\} \quad \textrm{for all $t \in (0,t_0)$ and $u \in (0,t_0)$.}
\end{equation*}
\end{lemma}

\begin{proof}
We have $f \in \mathscr{R}_0(\alpha)$ if and only if $x \mapsto f(1/x)$ is regularly varying at infinity with index $-\alpha$ (see \cite[pp.\ 17--18]{Bingham-Goldie-Teugels-1987}). The assertion follows then from the Potter's bounds for regular variation at infinity \cite[Theorem 1.5.6(iii)]{Bingham-Goldie-Teugels-1987}.
\end{proof}

\begin{proof}[Proof of Lemma \ref{lem:conv}]
We can write
\begin{equation*}
(f * g)(t) = \int_0^t f(t-u) g(u) \ud u = \int_0^1 t f\big(t(1-s)\big) g(ts) \ud s, \quad t > 0,
\end{equation*}
where we have made the substitution $s \eqdefl u/t$. Thus,
\begin{equation*}
\frac{(f * g)(t)}{t f(t) g(t)} = \int_0^1 \frac{f\big(t(1-s)\big)}{f(t)} \frac{g(ts)}{g(t)} \ud s, \quad t > 0,
\end{equation*}
where the integrand satisfies, for any $s \in (0,1)$,
\begin{equation}\label{pointwise}
\frac{f\big(t(1-s)\big)}{f(t)} \frac{g(ts)}{g(t)} \xrightarrow[t \rightarrow 0+]{} (1-s)^\alpha s^\beta.
\end{equation}

Using Lemma \ref{potter}, we can find for any $\varepsilon \in (0,\alpha+1)$ and $\delta \in (0,\beta+1)$, a threshold $t_0>0$ such that
\begin{align*}
\frac{f\big(t(1-s)\big)}{f(t)} & \leqslant 2 \max \big\{(1-s)^{\alpha-\varepsilon},(1-s)^{\alpha+\varepsilon}\big\}\leqslant 2 (1-s)^{(\alpha-\varepsilon)_-}, \\
\frac{g(ts)}{g(t)} & \leqslant 2 \max \big\{s^{\beta-\delta},s^{\beta+\delta} \big\} \leqslant 2 s^{(\beta-\delta)_-},
\end{align*}
for all $t \in (0,t_0)$ and $s \in (0,1)$, where $x_- \eqdefl \min\{x,0\}$ for all $x \in \R$. This yields the dominant
\begin{equation*}
\frac{f\big(t(1-s)\big)}{f(t)} \frac{g(ts)}{g(t)} \leqslant 4 s^{(\beta-\delta)_-}(1-s)^{(\alpha-\varepsilon)_-},
\end{equation*}
valid for all $t \in (0,t_0)$ and $s \in (0,1)$, which satisfies
\begin{equation*}
4 \int_0^1 s^{(\beta-\delta)_-}(1-s)^{(\alpha-\varepsilon)_-} \ud s = 4\, \mathrm{B}\big(1+(\beta-\delta)_- ,1+(\alpha-\varepsilon)_- \big)< \infty,
\end{equation*}
since $\alpha-\varepsilon>-1$ and $\beta-\delta > -1$.
The asserted convergence \eqref{rv-conv} follows now from \eqref{pointwise} by Lebesgue's dominated convergence theorem. 

Finally, $f * g \in \mathscr{R}_0(\alpha+\beta+1)$, since $f * g \geqslant 0$ and for any $\ell>0$,
\begin{equation*}
\frac{(f * g)(\ell t)}{(f * g)(t)} = \frac{(f * g)(\ell t)}{\ell t f(\ell t) g(\lambda t)} \frac{\ell f(\ell t) g(\ell t)}{f(t) g(t)} \frac{t f(t) g(t)}{(f * g)(t)} \xrightarrow[t \rightarrow 0+]{} \ell^{\alpha+\beta+1},
\end{equation*}
as implied by the limits
\begin{align*}
\lim_{t \rightarrow 0+}\frac{(f * g)(\ell t)}{\ell t f(\ell t) g(\ell t)} & = \mathrm{B}(\alpha+1,\beta+1), & \lim_{t \rightarrow 0+} \frac{\ell  f(\ell t) g(\ell t)}{ f(t) g(t)} & = \ell^{\alpha+\beta+ 1},\\
\lim_{t \rightarrow 0+} \frac{t f(t) g(t)}{(f * g)(t)} &= \frac{1}{\mathrm{B}(\alpha+1,\beta+1)},
\end{align*}
which follow from \eqref{rv-conv} and from the definition of regular variation at zero.
\end{proof}

\section{Conditional small ball property and hitting times}\label{app:hitting-times}

Bruggeman and Ruf \cite{Bruggeman-Ruf-2015} have recently studied the ability of a one-dimensional diffusion to hit arbitrarily fast any point of its state space. We remark that a similar property can be deduced for possibly non-Markovian processes directly from the CSBP. More precisely, in the multivariate case, any process that has the CSBP is able to hit arbitrarily fast any (non-empty) open set in $\R^d$ with positive conditional probability, even after any stopping time. While the following result is similar in spirit to some existing results in the literature (see, e.g., \cite[Lemma A.2]{Guasoni-Rasonyi-Schachermayer-2008}), it is remarkable enough that it deserves to be stated (and proved) here in a self-contained fashion.

\begin{prop}[Hitting times, multivariate case]\label{prop:hitting-times}
Suppose that $(\bs{Y}_t)_{t \in [0,T]}$ has the CSBP with respect to some filtration $(\mathscr{F}_t)_{t\in [0,T]}$. Then for any stopping time $\tau$ such that $\prob[\tau < T]>0$ and for any non-empty open set $A \subset \R^d$, the stopping time
\begin{equation*}
H^A_{\tau} \eqdefl \inf \{ t \in [\tau,T] : \bs{Y}_t \in A\}
\end{equation*}
satisfies for any $\delta > 0$,
\begin{equation*}
\prob\big[H^A_{\tau} < \tau + \delta \, \big| \, \mathscr{F}_{\tau}\big] >0 \quad \textrm{almost surely on $\{ \tau < T \}$.}
\end{equation*}
\end{prop}

\begin{proof}
Let $\delta>0$ and let $E \in \mathscr{F}_{\tau}$ be such that $E \subset \{ \tau < T \}$ and $\prob[E]>0$. Clearly, it suffices to show that
\begin{equation*}
\begin{split}
\E\Big[\mathbf{1}_E \prob\big[H^A_{\tau} < \tau + \delta \, \big| \, \mathscr{F}_{\tau}\big]\Big] & = \E\Big[\mathbf{1}_E \mathbf{1}_{\{H^A_{\tau} < \tau+\delta \}}\Big] \\
& = \prob\big[E \cap \big\{H^A_{\tau} < \tau+\delta \big\}\big]> 0.
\end{split}
\end{equation*}
By the assumption that the set $A$ is non-empty and open in $\R^d$, there exist $\bs{x}_A \in \R^d$ and $\varepsilon>0$ such that $B(\bs{x}_A,\varepsilon) \subset A$.
We can write $E = \bigcup_{\bs{z} \in \mathbb{Q}^d}\bigcup_{q \in [0,T) \cap \mathbb{Q}} E_{q,\bs{z}}$ with
\begin{equation*}
E_{q,\bs{z}} \eqdefl E \cap \bigg\{ q-\frac{\delta}{2}< \tau \leqslant q\bigg\} \cap \bigg\{ \| \bs{Y}_q - \bs{z}\| < \frac{\varepsilon}{2}\bigg\} \in \mathscr{F}_q.
\end{equation*}
It follows then that $\prob[E_{q_0,\bs{z}_0}]>0$ for some $q_0 \in [0,T) \cap \mathbb{Q}$ and $\bs{z}_0 \in \mathbb{Q}^d$, as $\prob[E]>0$. 

Define now $\bs{f} \in C_{\bs{0}}([q_0,T],\R^d)$ by
\begin{equation*}
\bs{f}(t) \eqdefl \frac{t-q_0}{T'-q_0} (\bs{x}_A-\bs{z}_0), \quad t \in [q_0,T],
\end{equation*}
where $T' \eqdefl \min \{ q_0 + \frac{\delta}{2},T\}$. Since
\begin{equation*}
\begin{split}
\|\bs{Y}_{T'} -\bs{x}_A\|  & \leqslant \| \bs{Y}_{T'}- \bs{Y}_{q_0} - \bs{f}(T') \| + \| \bs{Y}_{q_0} - \bs{z}_0\| \\
& \leqslant \sup_{t \in [q_0,T]}\| \bs{Y}_{t}- \bs{Y}_{q_0} - \bs{f}(t) \| + \| \bs{Y}_{q_0} - \bs{z}_0\|,
\end{split}  
\end{equation*}
we find that $\|\bs{Y}_{T'} -\bs{x}_A\| < \varepsilon$ on $E' \eqdefl E_{q_0,\bs{z}_0} \cap \{ \sup_{t \in [q_0,T]}\| \bs{Y}_{t}- \bs{Y}_{q_0} - \bs{f}(t) \|<\frac{\varepsilon}{2}\}$ and, a fortiori, that
\begin{equation*}
H^A_\tau \leqslant T' \leqslant q_0 + \frac{\delta}{2} < \tau + \delta \quad \textrm{on $E'$.}
\end{equation*}
Finally, as $E' \subset E$, we have
\begin{equation*}
\prob\big[E \cap \big\{H^A_{\tau} < \tau+\delta \big\}\big] \geqslant \prob[E'] \geqslant \E\bigg[\mathbf{1}_{E_{q_0,\bs{z}_0}}\prob\bigg[ \sup_{t \in [q_0,T]}\| \bs{Y}_{t}- \bs{Y}_{q_0} - \bs{f}(t) \|<\frac{\varepsilon}{2} \, \bigg| \, \mathscr{F}_{q_0}\bigg] \bigg] >0,
\end{equation*}
where the ultimate inequality follows from the CSBP.
\end{proof}

In the univariate case, a continuous process with CFS is able to hit any point in $\R$ arbitrarily fast. This is a straightforward corollary of Proposition \ref{prop:hitting-times}.

\begin{corollary}[Hitting times, univariate continuous case]
Suppose that a univariate continuous process $(Y_t)_{t \in [0,T]}$ has CFS with respect to some filtration $(\mathscr{F}_t)_{t\in [0,T]}$. Then for any stopping time $\tau$ such that $\prob[\tau < T]>0$ and for any $x \in \R$, the stopping time
\begin{equation*}
H^x_{\tau} \eqdefl \inf \{ t \in [\tau,T] : Y_t = x\}
\end{equation*}
satisfies for any $\delta > 0$,
\begin{equation*}
\prob\big[H^x_{\tau} < \tau + \delta \, \big| \, \mathscr{F}_{\tau}\big] >0 \quad \textrm{almost surely on $\{ \tau < T \}$.}
\end{equation*}
\end{corollary}


\end{document}